\documentclass[platex]{amsart}
\usepackage{amssymb}
\usepackage{braket}
\usepackage{mathrsfs}
\usepackage{ifthen}
\usepackage{graphicx}
\usepackage{tikz}
\usetikzlibrary{patterns}
\usepackage{comment} 
\usepackage{mleftright}
\usepackage{cleveref}
\usepackage[all]{xy}
\usepackage{amscd}
\usepackage{amsrefs}

\theoremstyle{plain}
\newtheorem{theo}{Theorem}[section]
\crefname{theo}{Theorem}{Theorems}
\Crefname{theo}{Theorem}{Theorems}

\crefname{prop}{Proposition}{Propositions}
\Crefname{prop}{Proposition}{Propositions}
\newtheorem{lem}[theo]{Lemma}
\crefname{lem}{Lemma}{Lemmas}
\Crefname{lem}{Lemma}{Lemmas}
\newtheorem{cor}[theo]{Corollary}
\crefname{cor}{Corollary}{Corollaries}
\Crefname{cor}{Corollary}{Corollaries}

\crefname{claim}{Claim}{Claims}
\Crefname{claim}{Claim}{Claims}

\crefname{property}{Property}{Properties}
\Crefname{property}{Property}{Properties}

\crefname{problem}{Problem}{Problems}
\Crefname{problem}{Problem}{Problems}

\theoremstyle{definition}

\crefname{defi}{Definition}{Definitions}
\Crefname{defi}{Definition}{Definitions}

\crefname{notation}{Notation}{Notations}
\Crefname{notation}{Notation}{Notations}

\crefname{convention}{Convention}{Conventions}
\Crefname{convention}{Convention}{Conventions}

\crefname{cond}{Condition}{Conditions}
\Crefname{cond}{Condition}{Conditions}

\crefname{assum}{Assumption}{Assumptions}
\Crefname{assum}{Assumption}{Assumptions}

\theoremstyle{remark}
\newtheorem{rem}[theo]{Remark}
\crefname{rem}{Remark}{Remarks}
\Crefname{rem}{Remark}{Remarks}
\newtheorem{ex}[theo]{Example}
\crefname{ex}{Example}{Examples}
\Crefname{ex}{Example}{Examples}

\crefname{section}{Section}{Sections}
\Crefname{section}{Section}{Sections}
\crefname{subsection}{Subsection}{Subsections}
\Crefname{subsection}{Subsection}{Subsections}
\crefname{figure}{Figure}{Figures}
\Crefname{figure}{Figure}{Figures}

\newtheorem*{acknowledgement}{Acknowledgement}

\newcommand{\Z}{\mathbb{Z}}

\newcommand{\R}{\mathbb{R}}
\newcommand{\C}{\mathbb{C}}
\newcommand{\quat}{\mathbb{H}}
\newcommand{\Q}{\mathbb{Q}}
\newcommand{\CP}{\mathbb{CP}}
\newcommand{\F}{\mathbb{F}}

\newcommand{\pt}{\mathrm{pt}}
\newcommand{\Pin}{\mathrm{Pin}}

\newcommand{\fraks}{\mathfrak{s}}
\newcommand{\frakt}{\mathfrak{t}}

\newcommand{\Diff}{\mathrm{Diff}}
\newcommand{\Homeo}{\mathrm{Homeo}}
\newcommand{\Aut}{\mathrm{Aut}}

\newcommand{\inc}{\hookrightarrow}

\newcommand{\del}{\partial}

\newcommand{\Ker}{\mathop{\mathrm{Ker}}\nolimits}
\newcommand{\Coker}{\mathop{\mathrm{Coker}}\nolimits}
\newcommand{\im}{\mathop{\mathrm{Im}}\nolimits}
\newcommand{\rank}{\mathop{\mathrm{rank}}\nolimits}

\newcommand{\Th}{\mathop{\mathrm{Th}}\nolimits}

\newcommand{\tilR}{\tilde{\mathbb{R}}}

\newcommand{\id}{\mathrm{id}}
\newcommand{\ind}{\mathop{\mathrm{ind}}\nolimits}

\newcommand{\frk}{\mathfrak}
\newcommand{\ks}{\mathop{\mathrm{ks}}\nolimits}

\def\t{\text}
\def\C{\mathbb{C}}

\def\D{\not\!\partial}

\def\wt{\widetilde}

\newcommand{\G}{\mathcal G}

\newcommand{\pr}{\text{pr}}
\newcommand{\Om}{\Omega}

\title[Diffeomorphisms and homeomorphisms of 4-manifolds with boundary]{The groups of diffeomorphisms and homeomorphisms of 4-manifolds with boundary}

\author{Hokuto Konno}
\address{Graduate School of Mathematical Sciences, the University of Tokyo, 3-8-1 Komaba, Meguro, Tokyo 153-8914, Japan}
\email{konno@ms.u-tokyo.ac.jp}

\author{Masaki Taniguchi}
\address{2-1 Hirosawa, Wako, Saitama 351-0198, Japan}
\email{masaki.taniguchi@riken.jp}

\begin{document}

\maketitle

\begin{abstract}
We give constraints on smooth families of 4-manifolds with boundary using Manolescu's Seiberg--Witten Floer stable homotopy type, provided that the fiberwise restrictions of the families to the boundaries are trivial families of 3-manifolds.
As an application, we show that, for a simply-connected oriented compact smooth 4-manifold $X$ with boundary with an assumption on the Fr{\o}yshov invariant or the Manolescu invariants $\alpha, \beta, \gamma$ of $\partial X$, 
the inclusion map $\mathrm{Diff}(X,\partial) \hookrightarrow \mathrm{Homeo}(X,\partial)$ between the groups of diffeomorphisms and homeomorphisms which fix the boundary pointwise is not a weak homotopy equivalence.
This combined with a classical result in dimension 3 implies that the inclusion map $\mathrm{Diff}(X) \hookrightarrow \mathrm{Homeo}(X)$ is also not a weak homotopy equivalence under the same assumption on $\del X$.
Our constraints generalize both of constraints on smooth families of closed 4-manifolds proven by Baraglia and a Donaldson-type theorem for smooth 4-manifolds with boundary originally due to Fr{\o}yshov.
\end{abstract}

\tableofcontents

\section{Introduction}
\label{section: Introduction}

The main purpose of this paper is to give constraints on smooth families of 4-manifolds with boundary using Manolescu's Seiberg--Witten Floer stable homotopy type~\cite{Ma03}, provided that the fiberwise restrictions of the families to the boundaries are trivial families of $3$-manifolds.
As an application, we show that, for a simply-connected oriented compact smooth 4-manifold $X$ with boundary with an assumption on the Fr{\o}yshov invariant of $\partial X$, 
the inclusion map
\begin{align*}
\mathrm{Diff}(X,\partial) \hookrightarrow \mathrm{Homeo}(X,\partial)
\end{align*}
is not a weak homotopy equivalence, where $\mathrm{Diff}(X,\partial)$ and $\mathrm{Homeo}(X,\partial)$ denote the groups of diffeomorphisms and homeomorphisms which fix the boundary pointwise respectively.
When $X$ is spin, the assumption on $\del X$ may be replaced with a similar assumption described in terms of the Manolescu invariants $\alpha, \beta, \gamma$.
This result combined with a classical theorem in dimension 3 implies that the inclusion map $\mathrm{Diff}(X) \hookrightarrow \mathrm{Homeo}(X)$ between the whole groups of diffeomorphisms and homeomorphisms is also not a weak homotopy equivalence under the same assumption on $\del X$.

Our constraints on smooth families of 4-manifolds with boundary have two roots.
The first is a constraint on smooth families of closed 4-manifolds proven by Baraglia~\cite{Ba19}, which can be regarded as a family version of Donaldson's diagonalization theorem.
The second is a constraint on negative-definite smooth $4$-manifolds with boundary originally due to Fr{\o}yshov~\cite{Fr96}, which is a generalization of Donaldson's diagonalization theorem  to $4$-manifolds with boundary.
Roughly speaking, our constraints are combinations of these two.

Let us recall some background of Baraglia's work.
It is classically known that, for a smooth closed manifold of dimension $< 4$, the natural inclusion map from the group of diffeomorphisms into the group of homeomorphisms is a weak homotopy equivalence.
However, in contrast, there are large numbers of examples of manifolds of dimension $\geq 4$ for which the above inclusion maps are not weak homotopy equivalences.
In dimension 4, the lowest dimension where such interesting difference happens, many authors revealed that gauge theory for families provides a strong tool to detect such phenomena.
See for example \cite{Ru98,KKN19,BK18,Ba19,KM20}.
In particular, Baraglia~\cite{Ba19} recently proved that
the inclusions from the diffeomorphism groups into the homeomorphism groups are not weak homotopy equivalences for a huge class of simply-connected closed smooth $4$-manifolds.
This is one of the important ingredients of this paper.

It is natural to try to extend Baraglia's result to $4$-manifolds with boundary.
He obtained his result by giving a constraint on smooth families of closed 4-manifolds, which is a family version of Donaldson's diagonalization theorem as mentioned above.
So a natural way to extend Baraglia's result is to obtain a constraint on smooth families of 4-manifolds with boundary.
We shall carry this out based on an idea of Fr{\o}yshov~\cite{Fr96}.
Although Fr{\o}yshov used monopole Floer homology to derive his constraint, we shall use Manolescu's Seiberg--Witten Floer stable homotopy type.
This is because Baraglia's argument is based on Furuta's idea of finite-dimensional approximation of the Seiberg--Witten equations~\cite{Fu01}, more precisely a family version of the Bauer--Furuta invariant~\cite{BF04}, and therefore we need to consider a family version of the relative Bauer--Furuta invariant, which lives in the Seiberg--Witten Floer stable homotopy type as far as the fiberwise restriction of a given family to the boundary is a trivial family of $3$-manifolds.

To state our main theorem, let us introduce some notations.
In this paper, we shall consider an oriented compact smooth $4$-manifold $X$ with boundary.
Throughout the paper, we shall assume that $b_{1}(X)=0$, and that $\del X = Y$ is a connected oriented rational homology $3$-sphere for simplicity.
As the structure groups of families of $X$, we have three candidates:
\begin{align*}
\Diff(X),\quad \Diff^{+}(X),\quad \Diff(X,\del).
\end{align*}
Here $\Diff(X)$ is the whole group of diffeomorphisms, and $\Diff^{+}(X)$ denotes the group of orientation-preserving diffeomorphisms, and $\Diff(X,\del)$ is the group of diffeomorphisms which fix the boundary pointwise.
Note that any element of $\Diff(X,\del)$ preserves the orientation of $X$.
Note also that, if the signature of $X$ is not zero, we have $\Diff(X) = \Diff^{+}(X)$.
We mainly consider $\Diff(X,\del)$ in this paper.
Similarly, we may define
\begin{align*}
\Homeo(X),\quad \Homeo^{+}(X),\quad \Homeo(X,\del)
\end{align*}
as the corresponding groups of homeomorphisms.
If a spin$^{c}$ structure or a spin structure $\fraks$ is given on $X$, one can define topological groups
\[
\Aut(X,\fraks),\quad \Aut((X,\fraks), \del).
\]
See \cref{rem: spinc wo metric} for the precise definition, but roughly $\Aut(X,\fraks)$ denote the automorphism group of the spin$^{c}$ (or spin) $4$-manifold $(X,\fraks)$, and $\Aut((X,\fraks), \del)$ is the structure group of families of spin$^{c}$ (or spin) $4$-manifolds where trivializations are given for the families of spin$^{c}$ (or spin) $3$-manifolds obtained as the boundaries.

Let $X \to E \to B$ be a $\Homeo(X, \del)$-bundle over a compact topological space $B$.
Then we have an associated vector bundle
\[
\R^{b^{+}(X)} \to H^{+}(E) \to B,
\]
whose isomorphism class is a topological invariant of $E$.
We shall explain $H^{+}(E)$ at the beginning of \cref{subsection: families relative Bauer--Furuta invariant}, but roughly speaking $H^{+}(E)$ is a bundle of maximal-dimensional positive-definite subspaces of $H^{2}$ of the fibers of $E$.
Our constraints on smooth families will be described in terms of $H^{+}(E)$.

For a rational homology $3$-sphere $Y$ with a spin$^{c}$ structure $\frakt$,
we denote by $\delta(Y,\frakt) \in \Q$ the Fr{\o}yshov invariant.
If $Y$ is an integral homology $3$-sphere, we denote by $\delta(Y)$ the Fr{\o}yshov invariant for the unique spin$^{c}$ structure on $Y$.
The sign convention of $\delta$ in this paper is $\delta(\Sigma(2,3,5))=1$, which is the same as the convention of \cite{Ma16}.
More precisely, we use $\delta$ defined by using $\F=\Z/2$-coefficient Seiberg--Witten Floer homology, which is denoted by $\delta_{2}$ in \cite{Ma16}.
(The reason why we use $\F$-coefficient is explained in \cref{rem: why not local system}.)

Now we can state the first main theorem in this paper:

\begin{theo}
\label{theo: main theo}
Let $Y$ be an oriented rational homology $3$-sphere and $X$ be an oriented compact smooth $4$-manifold bounded by $Y$.
Assume that $b_{1}(X)=0$.
Let $\fraks$ be a spin$^{c}$ structure on $X$ and let $\frakt$ be the spin$^{c}$ structure on $Y$ defined as the restriction of $\fraks$.
Let $B$ be a compact topological space and $(X,\fraks) \to E \to B$ a smooth $\Aut((X,\fraks), \del)$-bundle.
If the $b^{+}(X)$-th Stiefel-Whitney class satisfies $w_{b^{+}(X)}(H^{+}(E)) \neq 0$, then we have
\begin{align}
\label{eq: main ineq in main thm}
\frac{c_{1}(\fraks)^{2} - \sigma(X)}{8} \leq \delta(Y,\frakt).
\end{align}
\end{theo}

\Cref{theo: main theo} is an analog of Baraglia's constraint~\cite[Theorem 1.1]{Ba19} for families of spin$^{c}$ $4$-manifolds with boundary.
For the case that $B=\{\pt\}$, \cref{theo: main theo} recovers a special case of the constraint due to Fr{\o}yshov~\cite{Fr96} on the intersection form of a negative-definite smooth $4$-manifold with boundary.

For spin $4$-manifolds with boundary, we have a refinement of \cref{theo: main theo} using the Manolescu invariants $\alpha, \beta, \gamma$ defined in \cite{Ma16}, instead of $\delta$:

\begin{theo}
\label{theo: main theo2}
Let $Y$ be an oriented rational homology $3$-sphere and $X$ be an oriented compact smooth $4$-manifold bounded by $Y$.
Assume that $b_{1}(X)=0$.
Let $\fraks$ be a spin structure on $X$ and let $\frakt$ be the spin structure on $Y$ defined as the restriction of $\fraks$.
Let $B$ be a compact topological space and $(X,\fraks) \to E \to B$ a smooth $\Aut((X,\fraks), \del)$-bundle.
Then:
\begin{itemize}
\item If $w_{b^{+}(X)}(H^{+}(E)) \neq 0$ holds, then we have
\begin{align}
\label{eq: main ineq in main thm2}
\frac{-\sigma(X)}{8} \leq \gamma(Y,\frakt).
\end{align}
\item If $b^{+}(X) >0$ and $w_{b^{+}(X)-1}(H^{+}(E)) \neq 0$ holds, then we have
\begin{align}
\label{eq: main ineq in main thm3}
\frac{-\sigma(X)}{8} \leq \beta(Y,\frakt).
\end{align}
\item If $b^{+}(X) > 1$ and $w_{b^{+}(X)-2}(H^{+}(E)) \neq 0$ holds, then we have
\begin{align}
\label{eq: main ineq in main thm4}
\frac{-\sigma(X)}{8} \leq \alpha(Y,\frakt).
\end{align}
\end{itemize}
\end{theo}

\Cref{theo: main theo2} is an analog of Baraglia's constraint~\cite[Theorem 1.2]{Ba19} for families of closed spin $4$-manifolds with boundary.
For the case that $B=\{\pt\}$, F.~Lin \cite[Theorem~7]{FLin17} has proven these inequalities (for $X$ with two boundary components), which are extensions of Donaldson's Theorems~B and C to $4$-manifolds with boundary.

Using \cref{theo: main theo,theo: main theo2}, we may detect non-smoothable topological families of $4$-manifold with boundary, stated in \cref{theo: main app}.
As a consequence, we may detect homotopical difference between $\Diff(X,\del)$ and $\Homeo(X,\del)$ for a large class of $X$ as follows:

\begin{theo}
\label{cor: rel Diff Homeo}
Let $Y$ be an oriented integral homology $3$-sphere.
Let $X$ be a simply-connected, compact, oriented, smooth, and indefinite $4$-manifold with boundary 
$Y$.
Suppose that $\sigma(X) \leq 0$.
Suppose that $X$ and $Y$ satisfy at least one of the following conditions:
\begin{enumerate}
\item $\sigma(X)<-8$ and $\delta(Y) \leq 0$.
\item $\delta(Y) < 0$, and in addition $\sigma(X)<0$ if $X$ is non-spin.
\item $\sigma(X) = -8$, $\delta(Y) = 0$ and $\mu(Y)=1$, where $\mu(Y) \in \Z/2$ denotes the Rohlin invariant.
\item $X$ is spin and $-\sigma(X)/8 > \gamma(Y)$.
\item $X$ is spin, $b^{+}(X) > 1$ and $-\sigma(X)/8 > \beta(Y)$.
\item $X$ is spin, $b^{+}(X) > 2$ and $-\sigma(X)/8 > \alpha(Y)$.
\end{enumerate}
Then the inclusion map
\[
\Diff(X,\del) \inc \Homeo(X,\del)
\]
is not a weak homotopy equivalence.
\end{theo}

As a classical fact in dimension $3$, it is known that the groups of diffeomorphisms and homeomorphisms have no homotopical difference.
This combined with \cref{cor: rel Diff Homeo} implies a similar result also for $\Diff(X)$ and $\Homeo(X)$:

\begin{theo}
\label{cor: Diff Homeo absolute}
Let  $X$ and $Y$ be as in \cref{cor: rel Diff Homeo}.
Then the inclusion map
\[
\Diff(X) \inc \Homeo(X)
\]
is not a weak homotopy equivalence.
\end{theo}

In \cref{cor: rel Diff Homeo,cor: Diff Homeo absolute},
not just about weak homotopy equivalence, we may actually estimate the range of the degrees of homotopy groups where the difference happens for the first time:
it is approximately up to $b^{+}(X)$.
See \cref{cor: rel Diff Homeo ap,cor: Diff Homeo absolute ap} for the precise statements.

\begin{rem}
\label{rem: comparison assumption}
If $X$ is spin, the assumption (4)  in \cref{cor: rel Diff Homeo} is satisfied if we have that
\[
-\sigma(X)/8 > \delta(Y).
\]
This is deduced from a result by Stoffregen:
he showed in \cite[Theorem~1.2]{Sto172} that
\[
\alpha(Y,\frakt) \geq \delta(Y,\frakt) \geq \gamma(Y,\frakt)
\]
for a rational homology $3$-sphere $Y$ with a spin structure $\frakt$.

It is also worth noting that we have inequalities
\[
\alpha(Y,\frakt) \geq \beta(Y,\frakt) \geq \gamma(Y,\frakt),
\]
which follow from the definition of $\alpha, \beta, \gamma$.
\end{rem}

\begin{rem}
\label{rem: huge ex}
There are a huge (at least infinitely many) number of examples of $(X,Y)$ satisfying the assumption of \cref{cor: rel Diff Homeo}.
For example, it is quite easy to find examples satisfying (1) of \cref{cor: rel Diff Homeo}.
Other types of examples shall be given in \cref{subsec: Other applications and examples}.
The invariants $\alpha, \beta, \gamma, \delta$ are calculated by various authors, in particular for $\delta$ via an identification with the correction term in  Heegaard--Floer theory. 
See \cref{rem: comparison between Froyshov type invariants} for the details.
For $\alpha, \beta, \gamma$, see \cite[Subsection~3.8]{Ma16} and \cite{Sto172,Sto20}.
\end{rem}

For $X$ with small $b^{+}$, we can compare $\pi_{0}(\Diff(X,\del))$ with $\pi_{0}(\Homeo(X,\del))$ (and $\pi_{0}(\Diff(X))$ with $\pi_{0}(\Homeo(X))$ as well) a little more precisely:

\begin{theo}
\label{theo: not surj on alg aut}
Let $Y$ be an oriented integral homology $3$-sphere.
Let $X$ be a simply-connected, compact, oriented, smooth, and indefinite $4$-manifold with boundary $Y$.
Suppose that $\sigma(X) \leq 0$.
Suppose that $X$ and $Y$ satisfy at least one of the following conditions:
\begin{enumerate}
\item $b^{+}(X)=1$, $\sigma(X)<-8$ and $\delta(Y) \leq 0$.
\item $b^{+}(X)=1$, $\delta(Y) < 0$, and in addition $\sigma(X)<0$ if $X$ is non-spin.
\item $b^{+}(X)=1$, $\sigma(X) = -8$, $\delta(Y) = 0$ and $\mu(Y)=1$.
\item $b^{+}(X)=1$, $X$ is spin and $-\sigma(X)/8 > \gamma(Y)$.
\item $b^{+}(X)=2$, $X$ is spin and $-\sigma(X)/8 > \beta(Y)$.
\item $b^{+}(X)=3$, $X$ is spin and $-\sigma(X)/8 > \alpha(Y)$.
\end{enumerate}
Then the natural map
\begin{align}
\label{eq: pizero del}
\pi_{0}(\Diff(X,\del)) \to \pi_{0}(\Homeo(X,\del))
\end{align}
induced from the inclusion is not a surjection.

Moreover, the map
\begin{align}
\label{eq: pi zero closed}
\pi_{0}(\Diff(X)) \to \pi_{0}(\Homeo(X))
\end{align}
is also not a surjection.
Namely, there exists a homeomorphism of $X$ which is not
topologically isotopic to any self-diffeomorphism of $X$.
\end{theo}

\Cref{theo: not surj on alg aut} shall be proven in \cref{subsec: Comparison between Diff and Homeo}.
Concrete examples of $X$ satisfying the assumption of \cref{theo: not surj on alg aut} shall be given in \cref{ex: first first ex,ex: Ex of non surj1.5,ex: Ex of non surj2 rev,ex: Ex of non surj2,ex: Ex of non surj3}, where we shall use all invariants $\alpha, \beta,\gamma,\delta$.

\begin{rem}
As an obvious consequence of \cref{theo: not surj on alg aut}, in the setting of the \lcnamecref{theo: not surj on alg aut}, the natural map
\[
\pi_{0}(\Diff^{+}(X)) \to \Aut(H^{2}(X;\Z))
\]
is also not a surjection.
Here $\Aut(H^{2}(X;\Z))$ denotes the automorphism group of the intersection form.
\end{rem}

It is worth noting that, for a closed smooth $4$-manifold $X$, the map \eqref{eq: pi zero closed}
is often a surjection and there are only few examples of $X$ for which \eqref{eq: pi zero closed} are known to be not surjections.
See \cref{rem: non surj closed} for the detail.

Lastly, we mention that there are interesting recent work on relative diffeomorphisms in dimension $4$ based on techniques which are different from gauge theory.
See, for example, \cite{Wa18,BG19,Wa20}.

We finish off this introduction with an outline of the contents of this paper.
In \cref{section: Preliminaries} we summarize what we need regarding Manolescu's Seiberg--Witten Floer stable homotopy type.
In particular, in \cref{subsection: The Froyshov invariant and the Manolescu invariants} we recall some basics of the Fr{\o}yshov-type invariants $\alpha, \beta,\gamma,\delta$, and in \cref{subsection: families relative Bauer--Furuta invariant} we describe the families relative Bauer--Furuta invariant, from which we extract constraints on smooth families of $4$-manifolds with boundary, \cref{theo: main theo,theo: main theo2}.
In \cref{section: Proof of the main theorem} we prove \cref{theo: main theo,theo: main theo2}, which are the main theorems of this paper.
In \cref{section: Applications} we consider applications of \cref{theo: main theo,theo: main theo2} mainly to the existence of non-smoothable families of 4-manifolds with boundary, stated as \cref{theo: main app}, and give consequences of \cref{theo: main app} about comparisons between various diffeomorphism groups and homeomorphism groups of 4-manifolds with boundary in \cref{subsec: Comparison between Diff and Homeo}.
A number of examples of such comparison results are given in \cref{subsec: Other applications and examples}, where all of invariants $\alpha, \beta,\gamma,\delta$ are effectively used.

\begin{acknowledgement}
First the authors would like to express their gratitude to Tadayuki Watanabe for inspiring them to consider the group of relative diffeomorphisms by sharing a draft of his paper~\cite{Wa20} with them. 
The authors also wish to thank David Baraglia for giving helpful comments on previous versions of this paper.
The authors would also like to express their appreciation to Ciprian Manolescu and Nobuo Iida for answering them questions about Fr{\o}yshov-type invariants and a gauge fixing condition respectively.
The first author was partially supported by JSPS KAKENHI Grant Numbers 17H06461 and 19K23412.
The second author was supported by JSPS KAKENHI Grant Number 20K22319 and RIKEN iTHEMS Program.
\end{acknowledgement}

\section{Preliminaries} 
\label{section: Preliminaries}
In this \lcnamecref{section: Preliminaries}, we collect necessary ingredients to prove \cref{theo: main theo,theo: main theo2}.
After recalling the definition of Manolescu's Seiberg--Witten Floer stable homotopy type~\cite{Ma03} in \cref{subsection: Seiberg--Witten Floer stable homotopy type},
we recall some basics of the Fr{\o}yshov-type invariants $\alpha, \beta,\gamma,\delta$ in \cref{subsection: The Froyshov invariant and the Manolescu invariants}.
In \cref{subsection: families relative Bauer--Furuta invariant} we describe the families relative Bauer--Furuta invariant for a family of $4$-manifolds with boundary, defined if we suppose that the fiberwise restriction of the family to the boundaries is a trivial family of $3$-manifolds.
This is a main ingredient in the proof of \cref{theo: main theo,theo: main theo2}.

\subsection{Seiberg--Witten Floer stable homotopy type} 
\label{subsection: Seiberg--Witten Floer stable homotopy type}

In this subsection we review Manolescu's Seiberg--Witten Floer stable homotopy type, mainly to fix some notation.
The main references are Manolescu~\cite{Ma03} and Khandhawit~\cite{Kha15}. 

Let $(Y, \frakt)$ be an oriented spin$^c$ rational homology $3$-sphere with a Riemann metric $g_Y$. 
Let $S$ be the spinor bundle of $\frakt$.  Fix a flat spin$^{c}$ reference connection $a_0$ on $(Y, \frakt)$.
For an integer $k>2$, we define a configuration space by
\[
\mathcal{C}_k(Y,\frakt):=L^2_{k-\frac{1}{2}}  (i \Lambda^1_Y) \oplus  L^2_{k-\frac{1}{2}}(S).
\]
The {\it Chern--Simons--Dirac functional} $CSD : \mathcal{C}_k(Y,\frakt) \to \R$ is defined by
\[
CSD (a,\phi) :=  \frac{1}{2}   \left(-\int_Y a\wedge da + \int_Y <\phi , \D_{a_0+a} \phi >\t{dvol}_Y  \right), 
\]
where $\D_{a_0+a}$ is the $\text{spin}^c$ Dirac operator for the connection $a_0+a$.
This functional is invariant under the action of the gauge group, where the gauge group $\G_k(Y)$ and a subgroup $\G^{0}_k(Y)$ of $\G_k(Y)$ are defined by
\[
\G_k(Y) := L^2_{k+\frac{1}{2}} (Y, S^1)
\]
and
\[
\G^{0}_k(Y) := \Set{ g \in \G_k(Y)  | f : Y \to \R,\ g= e^{if},\ \int_{Y} f \text{vol}_Y =0 }.
\]
The action $\G_k(Y)$ on $\mathcal{C}_k(Y,\frakt)$ is given by the pull-back of connections and the complex multiplication on spinors.
A global slice of the action of $\G^{0}_{k}(Y) $ on $\mathcal{C}_k(Y,\frakt)$ is given by 
 \[
  \operatorname{Coul}_k(Y,\mathfrak{t})= (\Ker d^*:  L^2_{k-\frac{1}{2}}(\Lambda_Y^1) \to L^2_{k-\frac{3}{2}}(\Lambda_Y^0))  \oplus  L^2_{k-\frac{1}{2}}  ( S ). 
  \]
Let $\pr_{\Ker d^{*}} : L^2_{k-\frac{1}{2}}(\Lambda_Y^1) \to \Ker d^{\ast}$ be the linear projection with kernel the tangents to the $\G^{0}_{k}(Y)$-orbit.
The $S^1$-equivariant formal gradient flow
\[
\mathfrak{v} :  \operatorname{Coul}_k(Y,\mathfrak{t})  \to  \operatorname{Coul}_{k-1}(Y,\mathfrak{t})
\]
 of $CSD$ with respect to a certain metric on $\operatorname{Coul}_k(Y,\mathfrak{t})$ defined in \cite[Section~3]{Ma03} can be written as the sum of a linear part
 \[
 l= (*d,  \D_{a_0})
 \]
and the quadratic term
     \[
   c(b, \psi) = (\pr_{\Ker d^*}  \rho^{-1} ((\psi \psi^*)_0) , \rho (b) \psi- \xi (\psi) \psi), 
   \]
   where $\xi (\psi ) \in i \Om^0(Y)$ is determined by the conditions 
 \[
 d \xi (\psi) = ( 1- \pr_{\Ker d^*}  ) \circ \rho^{-1} ((\psi \psi^*)_0) \text{\quad and }  \ \int_Y \xi (\psi) \operatorname{dvol}  =0. 
 \]

 Henceforth we just say that $\frakt$ is spin if $\frakt$ comes from a spin structure.
Although $\frak{v}$ is an $S^1$-invariant vector field in general, if $\frakt$ is spin, we have a larger symmetry of the group $\Pin(2) $, which is defined by 
  \[
  \Pin(2) := S^1 \cup j S^1 \subset  Sp(1).
  \]
When $\frakt$ is spin, the spinor bundle has a structure of $Sp(1)$-bundle.
The group $\Pin(2)$ acts on the spinor bundle $S$ as the restriction of the natural $Sp(1)$-action on $S$, and $\Pin(2)$ acts on $\Om^1_Y$ via the non-trivial homomorphism $\Pin(2)\to O(1)=\{\pm1\}$. By such actions, $\Pin(2)$ acts on $\operatorname{Coul}_k(Y,\fraks)$.
It turns out that the vector field $\frk{v}$ is $\Pin(2)$-equivariant. 
Let $\tilR$ denote the real $1$-dimensional representation of $\Pin(2)$ via the map $\Pin(2)\to O(1)$, and $\quat$ denote the space of quaternions, on which $\Pin(2)$ acts as the restriction of the natural action of $Sp(1)$.

For $\lambda < 0 < \mu$, we define $V_\lambda^\mu(Y)$ as the direct sum of the eigenspaces whose eigenvalues belong to $(\lambda,\mu]$. Here we think of $V_\lambda^\mu(Y)$ as a subspace of $\operatorname{Coul}_k(Y,\frakt)$. We denote by
 \begin{align}
 p_\lambda^\mu: \operatorname{Coul}_k(Y,\frakt) \to V_\lambda^\mu(Y)
 \end{align}
 the $L^2$-projection of $\operatorname{Coul}_k(Y,\frakt)$ onto $V_\lambda^\mu(Y)$.
 Henceforth we often abbreviate $V_\lambda^\mu(Y)$ as $V_\lambda^\mu$.
 Since $l$ is the sum of a real operator and a complex operator, we have the direct sum decomposition 
 \[
 V_\lambda^\mu = V_\lambda^\mu(\R) \oplus V_\lambda^\mu(\mathbb{C})
 \]
 of a real vector space and a complex vector space. 
 Denote by $B(R; V_\lambda^\mu)$ the closed ball in $V_\lambda^\mu$ of radius $R$ centered at the origin.
 Manolescu proved the following compactness property for the dynamical system induced by a vector field $(V_\lambda^\mu, l + p_\lambda^\mu c)$: 
 
\begin{theo}[{\cite[Proposition~3]{Ma03}}]
 \label{isolating}
 There exist sufficiently large $R>0$ and $-\lambda, \mu>0$ such that all trajectories $x : \R \to V_\lambda^\mu$ of the flow equation 
 \[
 \frac{\partial}{\partial t} x(t) = - (l + p_\lambda^\mu c) (x(t)) 
 \]
 which lie in $B(2R; V_\lambda^\mu)$ actually lie in $B(R; V_\lambda^\mu) $.
\end{theo}
 
 By the use of \cref{isolating}, one can see that $B(2R; V_\lambda^\mu)$ is an isolating neighborhood of 
 \[
 \operatorname{Inv} B(2R; V_\lambda^\mu) := \Set{ x \in B(2R) | t \cdot x \in B(2R), ~  \forall t \in \R } 
 \]
  with respect to the flow on $V_\lambda^\mu$ generated by $\rho (l + p_\lambda^\mu c)$, where $\rho$ is an $S^1$-invariant bump function such that $\rho|_{V_\lambda^\mu \setminus B(3R; V_\lambda^\mu)}=0$ and $\rho|_{B(2R; V_\lambda^\mu)}=1$.
Here $t\cdot$ denotes the action of $t$ via this flow on $V_{\lambda}^{\mu}$.
 We denote by 
 \begin{align}\label{Conley index} 
 I_\lambda^\mu := N_\lambda^\mu /  L_\lambda^\mu
 \end{align}
  the $S^1$-equivariant index pair $(N_\lambda^\mu ,  L_\lambda^\mu)$ of  $\operatorname{Inv} B(2R; V_\lambda^\mu)$ for the flow.
 Let $n(Y, \frakt, g_Y) \in \Q$ be the rational number defined as
 \[
 n(Y, \frakt, g_Y) = \ind_{\C}D_{X} - \frac{c_{1}(\fraks)^{2}-\sigma(X)}{8},
 \]
where $D_{X}$ is the Dirac operator on a compact oriented $4$-manifold $X$ bounded by $Y$ with respect to a spin$^{c}$ structure $\fraks$ on $X$ which is an extension of $\frakt$, and a spin$^{c}$ connection on $X$ which is an extension of the reference connection $a_{0}$ on $Y$.
It turns out that $n(Y, \frakt, g_Y)$ depends only on $(Y, \frakt, g_Y)$.
  The Seiberg--Witten Floer homotopy type of $(Y,\frakt)$ is defined to be
 \begin{align}
 SWF(Y, \frakt) := \Sigma^{-n(Y, \frakt, g_Y) \mathbb{C} -V^0_\lambda} I_\lambda^\mu,
 \end{align}
 which makes sense in a certain suspension category.
 For the definition of the formal desuspension, see \cite[Section~6]{Ma03}.
 (However, we shall use only $SWF(Y, \frakt)$ which is sufficiently suspended in that category, and  the formal desuspensions will not appear in our argument.)

When $\frakt$ is spin, we take $\rho$ above to be a $\Pin(2)$-invariant bump function, and consider $\Pin(2)$-equivariant Conley index instead.
We set 
\[
 SWF(Y, \frakt) :=  \Sigma^{-\frac{n(Y, \frakt, g)}{2} \mathbb{H} -V^0_\lambda}  I^\mu_\lambda, 
 \]
as a stable homotopy type of a pointed $\Pin(2)$-space.

\subsection{The Fr{\o}yshov invariant $\delta$ and the Manolescu invariants $\alpha, \beta, \gamma$}
\label{subsection: The Froyshov invariant and the Manolescu invariants}

In this \lcnamecref{subsection: The Froyshov invariant and the Manolescu invariants} we recall the definition of the Fr{\o}yshov invariant and the Manolescu invariants $\alpha, \beta, \gamma$.
The Fr{\o}yshov invariant was originally defined in term of the monopole Floer homology \cite{Fr96,Fr02}, but it can be interpreted also in terms of the Seiberg--Witten Floer homotopy type~\cite{Ma03,Ma16}.
In this paper we mainly follow Manolescu's description of the Fr{\o}yshov invariant given in \cite{Ma16}.
When one considers a spin structure on a given $3$-manifold, using $\Pin(2)$-symmetry of the Seiberg--Witten equations, analogous three invariants are defined, which are the Manolescu invariants $\alpha, \beta, \gamma$ introduced in \cite{Ma16}.
We also recall the definition of $\alpha, \beta, \gamma$ in this \lcnamecref{subsection: The Froyshov invariant and the Manolescu invariants}.
Henceforth, throughout this paper, all (co)homology will be taken with $\F=\Z/2$-coefficients.
We refer the readers also to Stoffregen's paper~\cite{Sto172} for this \lcnamecref{subsection: The Froyshov invariant and the Manolescu invariants}.

\begin{rem}
The original definition of the Fr{\o}yshov invariant uses (co)homology with $\Q$-coefficient, not $\F$-coefficient.
However, as noted in \cite[Remark~3.12]{Ma16}, there is no known example of $3$-manifolds for which the Fr{\o}yshov invariant with $\Q$-coefficient and that with $\F$-coefficient are different.
\end{rem}

Let $(Y,\frakt)$ be a spin$^{c}$ rational homology 3-sphere and fix a Riemannian metric $g$ on $Y$ and real numbers  $\lambda, \mu$ to define a finite-dimensional approximation.
One can easily check that 
\[
(I_{\lambda}^{\mu})^{S^{1}} \cong N^{S^{1}}/L^{S^{1}}
\]
and $(I_{\lambda}^{\mu})^{S^{1}}$ is homotopy equivalent to $V_{\lambda}^{0}(\R)^{+}$, 
Set 
\[
s := \dim V_{\lambda}^{0}(\R).
\]
Then we have
\[
\tilde{H}_{S^{1}}^{\ast + s}((I_{\lambda}^{\mu})^{S^{1}})
\cong \tilde{H}_{S^{1}}^{\ast + s}(V_{\lambda}^{0}(\R)^{+})
\cong \tilde{H}_{S^{1}}^{\ast}(S^{0})
\cong \F[U].
\]
The Fr{\o}yshov invariant $\delta(Y,\frakt)$ is defined as follows.
Denote by $i : (I_{\lambda}^{\mu})^{S^{1}} \inc I_{\lambda}^{\mu}$ the inclusion.
The quantity $d$ in \cite{Ma16} is defined as 
\begin{align}
\label{eq: d def}
d(Y, \lambda, \mu, g, \frakt)
= \min\Set{r \equiv s \mod 2 | \exists x \in \tilde{H}_{S^{1}}^{r}(I_{\lambda}^{\mu}),\ U^{l} \cdot x \neq 0 \ (\forall l \geq 0)},
\end{align}
where an equivariant localization theorem ensures that the set in the right-hand side is not empty.
This might look different from the quantity $d$ defined in \cite[Subsection~2.6]{Ma16} at first glance, but it can be seen that \eqref{eq: d def} is just the same with Manolescu's $d$ using an argument in the proof of \cite[Lemma~2.9]{Ma16}.
(See also \cite[Definition~3.6]{Sto172}.)
Then the Fr{\o}yshov invariant $\delta(Y,\frakt) \in \Q$ is defined by
\begin{align}
\label{def Fro h d}
\delta(Y, \frakt)
= (d(Y, \lambda, \mu, g, \frakt) - \dim_{\R} V_{\lambda}^{0})/2 -n(Y,\frakt,g).
\end{align}
It turns out that $\delta(Y, \frakt) \in \Q$ depends only on $(Y, \frakt)$.
(Note that $n(Y,\frakt,g)$ may not be an integer if $Y$ is not an integral homology sphere.
If $Y$ is an integral homology sphere, then $n(Y,\frakt,g) \in \Z$ and hence $\delta(Y, \frakt) \in \Z$.)

Here we note an elementary observation used in the proof of one of the main theorem, \cref{theo: main theo}.

\begin{lem}
\label{lem: non zero red}
If $x \in \tilde{H}_{S^{1}}^{r}(I_{\lambda}^{\mu})$ satisfies that $U^{l} \cdot x \neq 0$ for all $l \geq 0$, then we have $i^{\ast} x \neq 0$ in $\tilde{H}_{S^{1}}^{r}((I_{\lambda}^{\mu})^{S^{1}}).$
\end{lem}

\begin{proof}
As well as \cite[Fact~2.5]{Sto172}, an equivariant localization theorem implies that
\begin{align}
\label{eq: Sto iso}
i^{\ast} : \tilde{H}_{S^{1}}^{\ast}(I_{\lambda}^{\mu}) \to \tilde{H}_{S^{1}}^{\ast}((I_{\lambda}^{\mu})^{S^{1}})
\end{align}
is an isomorphism in cohomology in sufficiently high degrees.
The map \eqref{eq: Sto iso} is a $\tilde{H}_{S^{1}}^{\ast}(S^{0})=\F[U]$-module map.
Thus we have $i^{\ast}U^{l} \cdot x = U^{l} \cdot i^{\ast}x$ for all $l \geq 0$.
Therefore it suffices to show that there exists $l \geq 0$ such that $i^{\ast}U^{l} \cdot x \neq 0$ to prove the \lcnamecref{lem: non zero red}.
However, if we take $l$ sufficiently large, $i^{\ast} : \tilde{H}_{S^{1}}^{4l+r}(I_{\lambda}^{\mu}) \to \tilde{H}_{S^{1}}^{4l+r}((I_{\lambda}^{\mu})^{S^{1}})$ is an isomorphism, and we have that $U^{l} \cdot x \neq 0$.
Thus we obtain $i^{\ast}U^{l} \cdot x \neq 0$ for sufficiently large $l$.
\end{proof}

\begin{lem}
\label{lem: expression of omega}
Set $d=d(Y, \lambda, \mu, g, \frakt)$.
Then there exists a cohomology class 
\[
\omega \in \tilde{H}_{S^{1}}^{d}(I_{\lambda}^{\mu})
\]
such that
\begin{align}
\begin{split}
\label{eq: expression of iomega}
i^{\ast}\omega = [V^{0}_{\lambda}(\R)^{+}] \otimes U^{(d-s)/2}
\end{split}
\end{align}
in
\[
\tilde{H}^{\ast}(V^{0}_{\lambda}(\R)^{+}) \otimes \tilde{H}_{S^{1}}^{\ast}(S^{0})
\cong \tilde{H}_{S^{1}}^{\ast}(V^{0}_{\lambda}(\R)^{+}) 
\cong \tilde{H}_{S^{1}}^{\ast}((I_{\lambda}^{\mu})^{S^{1}}).
\]
(Recall that $d \equiv s \mod 2$, and hence $U^{(d-s)/2}$ makes sense.)
\end{lem}

\begin{proof}
By the definition of $d$ given in \eqref{eq: d def} and \cref{lem: non zero red}, there exists a cohomology class $\omega \in \tilde{H}_{S^{1}}^{d}(I_{\lambda}^{\mu})$
such that $i^{\ast}\omega \neq 0$.
Notice that this non-vanishing of $i^{\ast}\omega$ is equivalent to \eqref{eq: expression of iomega}.
\end{proof}

Next we recall the definition of $\alpha,\beta,\gamma$.
Suppose that $\frakt$ comes from a spin structure.
Then we have
\begin{align}
\label{eq: cohomology ring Pin2}
\tilde{H}_{\Pin(2)}^{\ast + s}((I_{\lambda}^{\mu})^{S^{1}})
\cong \tilde{H}_{\Pin(2)}^{\ast + s}(V_{\lambda}^{0}(\tilR)^{+})
\cong \tilde{H}_{\Pin(2)}^{\ast}(S^{0})
\cong \F[q,v]/(q^{3}),
\end{align}
with elements $q$ in degree 1 and $v$ in degree 4.
Let us define
\begin{align*}
a(Y, \lambda, \mu, g, \frakt)
&= \min\Set{r \equiv s \mod 4 | \exists x \in \tilde{H}_{\Pin(2)}^{r}(I_{\lambda}^{\mu}),\ v^{l} \cdot x \neq 0 \ (\forall l \geq 0)},\\
b(Y, \lambda, \mu, g, \frakt)
&= \min\Set{r \equiv s +1 \mod 4 | \exists x \in \tilde{H}_{\Pin(2)}^{r}(I_{\lambda}^{\mu}),\ v^{l} \cdot x \neq 0 \ (\forall l \geq 0)}-1,\\
c(Y, \lambda, \mu, g, \frakt)
&= \min\Set{r \equiv s +2 \mod 4 | \exists x \in \tilde{H}_{\Pin(2)}^{r}(I_{\lambda}^{\mu}),\ v^{l} \cdot x \neq 0 \ (\forall l \geq 0)}-2.
\end{align*}
The definition of the invariants $\alpha,\beta,\gamma$ valued in $\Q$ is
\begin{align}
\alpha(Y, \frakt)
&= (a(Y, \lambda, \mu, g, \frakt) - \dim_{\R} V_{\lambda}^{0})/2 -n(Y,\frakt,g), \label{eq: def al}\\
\beta(Y, \frakt)
&= (b(Y, \lambda, \mu, g, \frakt) - \dim_{\R} V_{\lambda}^{0})/2 -n(Y,\frakt,g),\label{eq: def beta}\\
\gamma(Y, \frakt)
&= (c(Y, \lambda, \mu, g, \frakt) - \dim_{\R} V_{\lambda}^{0})/2 -n(Y,\frakt,g).\label{eq: def gamma}
\end{align}

\begin{lem}
\label{lem: non zero red spin}
If $x \in \tilde{H}_{\Pin(2)}^{r}(I_{\lambda}^{\mu})$ satisfies that $v^{l} \cdot x \neq 0$ for all $l \geq 0$, then we have $i^{\ast} x \neq 0$ in $\tilde{H}_{\Pin(2)}^{r}((I_{\lambda}^{\mu})^{S^{1}}).$
\end{lem}

\begin{proof}
The proof is totally similar to the proof of \cref{lem: non zero red}:
just use instead the fact, which is precisely \cite[Fact~2.5]{Sto172}, that
\begin{align*}
i^{\ast} : \tilde{H}_{\Pin(2)}^{\ast}(I_{\lambda}^{\mu}) \to \tilde{H}_{\Pin(2)}^{\ast}((I_{\lambda}^{\mu})^{S^{1}})
\end{align*}
is an isomorphism in sufficiently high degrees.
\end{proof}

\begin{lem}
\label{lem: expression of omega spin}
Set $a=a(Y, \lambda, \mu, g, \frakt), b=b(Y, \lambda, \mu, g, \frakt), c=c(Y, \lambda, \mu, g, \frakt)$.
Then there exist cohomology classes 
\begin{align*}
&\omega_{a} \in \tilde{H}_{S^{1}}^{a}(I_{\lambda}^{\mu}),\\
&\omega_{b} \in \tilde{H}_{S^{1}}^{b+1}(I_{\lambda}^{\mu}),\\
&\omega_{c} \in \tilde{H}_{S^{1}}^{c+2}(I_{\lambda}^{\mu})
\end{align*}
such that
\begin{align}
\begin{split}
\label{eq: expression of iomega spin}
&i^{\ast}\omega_{a} = \tau_{\Pin(2)}(V^{0}_{\lambda}(\tilR)^{+}) \cup v^{(a-s)/4},\\
&i^{\ast}\omega_{b} = \tau_{\Pin(2)}(V^{0}_{\lambda}(\tilR)^{+}) \cup qv^{(b-s)/4},\\
&i^{\ast}\omega_{c} = \tau_{\Pin(2)}(V^{0}_{\lambda}(\tilR)^{+}) \cup q^{2}v^{(c-s)/4}
\end{split}
\end{align}
in
\[
\tilde{H}_{S^{1}}^{\ast}(V^{0}_{\lambda}(\tilR)^{+}) 
\cong \tilde{H}_{S^{1}}^{\ast}((I_{\lambda}^{\mu})^{S^{1}}).
\]
Here $\tau_{\Pin(2)}(V^{0}_{\lambda}(\tilR)^{+}) \in \tilde{H}_{\Pin}^{*}(V^{0}_{\lambda}(\tilR)^{+})$ is the equivariant Thom class of the bundle $V^{0}_{\lambda}(\tilR) \to \pt$ over a point.
(Recall that $a,b,c$ are congruent to $s$ ${\rm mod}\ 4$,
and hence $v^{(a-s)/4},\ v^{(b-s)/4},\ v^{(c-s)/4}$ make sense.)
\end{lem}

\begin{proof}
By the definition of $a,b,c$ and \cref{lem: non zero red spin}, there exist cohomology classes $\omega_{a} \in \tilde{H}_{S^{1}}^{a}(I_{\lambda}^{\mu}),\ \omega_{b} \in \tilde{H}_{S^{1}}^{b+1}(I_{\lambda}^{\mu}),\ \omega_{c} \in \tilde{H}_{S^{1}}^{c+2}(I_{\lambda}^{\mu})$
whose pull-back under $i$ do not vanish.
Forgetting the degree shift by $s$ for the moment, the non-zero cohomology classes
$i^{\ast}\omega_{a}, i^{\ast}\omega_{b}, i^{\ast}\omega_{c}$ are of the form $v^{l}, qv^{l'}, q^{2}v^{l''}$ respectively by the degree reason: 
\[
a \equiv s \mod 4,\quad b+1 \equiv s+1 \mod 4,\quad c+2 \equiv s+2 \mod 4.
\]
Recalling that the degree shift by $s$ is occurred by multiplying the equivariant Thom class, we can determine $l,l',l''$ and obtain \eqref{eq: expression of iomega spin}.
\end{proof}

\subsection{The families relative Bauer--Furuta invariant} 
\label{subsection: families relative Bauer--Furuta invariant}

In this section we consider a family version of the relative Bauer--Furuta invariant. 

Let $X$ be an oriented compact smooth $4$-manifold bounded by $Y$.
Assume that $b_{1}(X) =0$ and $Y$ is a connected rational homology $3$-sphere.
Let $\fraks$ be a spin$^{c}$ structure on $X$ and let $\frakt$ be the spin$^{c}$ structure on $Y$ defined as the restriction of $\fraks$.
Let $B$ be a compact topological space.
Throughout this paper, for a topological space $F$, denote by $\underline{F}$ the trivialized bundle $B \times F$ over $B$.

Assume that we have a $\Homeo(X,\del)$-bundle $X \to E \to B$.
We shall define a vector bundle
\[
\R^{b^{+}(X)} \to H^{+}(E) \to B
\]
as follows.
First, let us define the `maximal-positive-definite Grassmannian' 
\[
\mathrm{Gr}^{+}(H^{2}(X;\R))
\]
as the space of maximal-dimensional positive-definite subspace of $H^{2}(X;\R)$ with respect to the intersection form.
Since the group $\Homeo(X,\del)$ naturally acts on $\mathrm{Gr}^{+}(H^{2}(X;\R))$, we obtain a fiber bundle 
\[
\mathrm{Gr}^{+}(H^{2}(X;\R)) \to \mathrm{Gr}^{+}_{E} \to B
\]
associated with $E$ with fiber $\mathrm{Gr}^{+}(H^{2}(X;\R))$.
The Grasmannian $\mathrm{Gr}^{+}(H^{2}(X;\R))$ is contractible,
since this is diffeomorphic to the quotient of the Lie group $O(b^{+}(X),b^{-}(X))$ by the maximal compact subgroup $O(b^{+}(X)) \times O(b^{-}(X))$.
Therefore there exists a section of $\mathrm{Gr}^{+}_{E} \to B$, which is unique up to isotopy.
One section  corresponds to a vector bundle of rank $b^{+}(X)$, and we denote by $H^{+}(E)$ the vector bundle.
This vector bundle is determined uniquely by $E$ up to isomorphism, and we omit the choice of section of $\mathrm{Gr}^{+}_{E} \to B$ from our notation $H^{+}(E)$.

\begin{rem}
\label{rem: spinc wo metric}
A spin structure on an oriented $n$-manifold for which a Riemannian metric is not given can be defined as a lift of the structure group of the frame bundle from $GL^{+}(n,\R)$ to the double cover $\widetilde{GL}^{+}(n,\R)$, where $GL^{+}(n,\R)$ is the group of real square matrixes of order $n$ of $\det >0$.
A spin$^{c}$ structure is also similarly defined using $(\widetilde{GL}^{+}(n,\R) \times S^{1})/\pm1$ instead of $Spin^{c}(n)$.
If a spin structure or a spin$^{c}$ structure $\fraks$ is given on $X$, let us define groups
\[
\Aut(X,\fraks),\quad \Aut((X,\fraks), \del)
\]
as follows.
First $\Aut(X,\fraks)$ denote the automorphism group of the spin or spin$^{c}$ $4$-manifold $(X,\fraks)$.
Namely, each element of $\Aut(X,\fraks)$ is a pair $(f,\tilde{f})$ of a diffeomorphism $f$ which preserves orientation and the isomorphism class of $\fraks$, and a lift $\tilde{f}$ of $f$ to a bundle automorhism of the principal $\widetilde{GL}^{+}(n,\R)$- or $(\widetilde{GL}^{+}(n,\R) \times S^{1})/\pm1$-bundle $P$ corresponding to $\fraks$.
The group $\Aut((X,\fraks), \del)$ is defined as the subgroup of $\Aut(X,\fraks)$ consisting of pairs $(f,\tilde{f})$ whose restrictions to $\del X$ and $P|_{\del X}$ are the pair of the identity maps.
\end{rem}

From here we assume that a reduction of $E$ to $\Aut((X,\fraks), \del)$ is given.
Namely, $(X,\fraks) \to E \to B$ is a smooth fiber bundle of spin$^{c}$ $4$-manifolds equipped with a trivialization
\[
 ( (Y ,\frakt ) \to E_Y \to B ) \cong  ( (Y ,\frakt ) \to ( Y , \frakt )\times B  \to B ),
 \]
 where $E_Y$ is a fiber bundle over $B$ defined to be
 \[
\bigsqcup_{b \in B} \partial E_b \to B. 
 \]
 
 Fix a fiberwise metric $g_E$ on $E \to B$ such that, near a color neighborhood $[-\varepsilon, 0 ] \times  \partial E_b$ of $\partial E_b$,
 \[
 g_E|_{[-\varepsilon, 0 ] \times  \partial E_b}  = \pi^{\ast}g_Y+ dt^2,
 \]
 where $g_Y$ is a fixed Riemann metric on $Y$ and $\pi : E_b \to \partial (E_b)=Y$ is the projection.
 Let $\{\widehat{A}_b\}_{b \in B}$ be a fiberwise reference spin$^c$-connection on $E$ such that $\widehat{A}_b|_{\partial E_b}=a_0$ for any $b \in B$.
Once we fix the data $(E, g_E)$, the following families of vector bundles over $B$
\[
S^+_E,\quad  S^-_E,\quad i \Lambda^*_E,\quad i\Lambda^+_E
\]
are associated.
The restrictions of them over $b \in B$ are the positive and negative spinor bundles with respect to $(g_{E_b}, \fraks)$, and $i\Lambda^*_{X}$ and  $i\Lambda^+_{X}$ with respect to $g_{E_b}$ respectively, where $\Lambda^+_{X}$ denotes the space of self-dual $2$-forms.
We use the notation
\[
L^2_k (S^+_E),\quad  L^2_k(S^-_E),\quad  L^2_k (i\Lambda^*_E),\quad  L^2_k  (i \Lambda^+_E)
\]
to denote the spaces of fiberwise $L^2_k$-sections.
In order to obtain the Fredholm property for a certain operator, we shall use a subspace $L^2_{k} ( i\Lambda^1_{E})_{CC}$ of $L^2_{k} ( i\Lambda^1_{E})$ defined by
\[
L^2_k (i\Lambda^1_E)_{CC} :=\bigsqcup_{b \in B}  \Set{ a \in L^2_k (i  \Lambda ^1_{E_b})   | d^* a =0,\  d^*{\bf t}a=0 }, 
\]
where ${\bf t}$ is the restriction as differential forms along the inclusion $Y= \partial E_b  \inc E_b$. 
This gauge fixing condition is called the {\it double Coulomb condition} and was introduced by Khandhawit~\cite{Kha15}.

\begin{rem}
Although Khandhawit imposed the condition $\int_Y {\bf t}(* a)=0 $, we can omit this condition. 
This is because we have
\[
\int_Y {\bf t}(* a)  = \int_Y{ \bf t}1 \wedge *{\bf n}a_b 
=  \int_{E_b} d1 \wedge *a_b  - \int_{E_b} 1 \wedge *d^*a_b =0
\]
by the Stokes theorem for any $a_b \in L^2_{k} ( i\Lambda^1)_{CC}$, where ${\bf n}$ is the normal component. 
Here we used the connectivity of $Y$. 
\end{rem}

For any positive real number $\mu$, now we have the fiberwise Seiberg--Witten map over a slice
\[
\mathcal{F}^\mu : L^2_{k} ( i\Lambda^1)_{CC}  \oplus L^2_k ( S^+_E)  \to     L^2_{k-1} ( i\Lambda^+)\oplus L^2_{k-1} ( S^-_E)  \oplus \underline{V^\mu_{-\infty}}
\]
defined by 
\[
\mathcal{F}^\mu ( (A_b, \Phi_b )_{b \in B} ) = (  \rho_b (F^+ (A_b)) -  (\Phi_b, \Phi_b)_0 , {D}_{\widehat{A}_b+ A_b} ( \phi),p^\mu_{-\infty} \circ r_b ( A_b, \Phi_b ) )_{b \in B} , 
\]
where $F^+ (A_b)$ is the self-dual part of the curvature of a fiberwise connection $A_b$, $\rho_b$ is the Clifford multiplication, ${D}_{\widehat{A}_b+ A_b}$ is the fiberwise Dirac operator with respect to a connection $\widehat{A}_b+ A_b$, and 
\[
r_b : L^2_{k} ( i\Lambda^1)_{CC}  \oplus L^2_k ( S^+_E) \to  \underline{\operatorname{Coul}_k(Y,\fraks)}
\]
 is the fiberwise restriction.
We decompose $\mathcal{F}^{\mu}$ as the sum of a fiberwise linear operator 
\begin{align}
L^{\mu} = 
\{L^\mu_b= (d^+, D_{\widehat{A}_b },  p^\mu_{-\infty} r_b )\}_{b \in B}
\end{align}
 and a fiberwise quadratic part
 \[
c^{\mu}
= \{ c^\mu_b  = (- (\Phi_b \otimes \Phi^*_b)_0, \rho (A_b) (\Phi_b), 0) \}_{b \in B}.
 \] 
 We often use a decomposition of the operator $L^\mu_b$ for each $b$ as the sum of the real operator 
 \[
 L^\mu_{b, \R}= (d^+, 0,  p^\mu_{-\infty}r_b ) : L^2_{k} ( i\Lambda^1_{E_{b}})_{CC} \to     L^2_{k-1} ( i\Lambda^+_{E_{b}})\oplus V^\mu_{-\infty}(\R)
 \]
 and the complex operator 
 \[
 L^\mu_{b, \C} =(0 , D_{\widehat{A}_b },  p^\mu_{-\infty}r_b ) :  L^2_k ( S^+_{E_{b}})  \to    L^2_{k-1} ( S^-_{E_{b}})  \oplus V^\mu_{-\infty}(\C). 
 \]
  
It is checked in \cite{Kha15} that the fiberwise linear operator ${L}^\mu_b$ is Fredholm on each fiber and the Fredholm index is given by 
\[
2\ind_{\mathbb{C}}^{APS} D^+_{A_b}  -b^+(X) - \dim V^\mu_0 , 
\]
 where $\ind_{\mathbb{C}}^{APS} D^+_{A_b}$ is the Fredholm index of $L^\mu_{b, \C}$ as a complex operator. 

The following lemma provides fundamental properties of the linear map $L^0_{b, \R} :  L^2_{k} ( i\Lambda^1_{E_{b}})_{CC}  \to   L^2_{k-1} ( i\Lambda^+_{E_{b}})\oplus V^0_{-\infty} (\R)$. 

\begin{lem}
\label{linear injection}
Under the assumption $b_1(X)=0$, 
the operator
\[
L^0_{b, \R} :  L^2_{k} ( i\Lambda^1_{E_{b}})_{CC}  \to   L^2_{k-1} ( i\Lambda^+_{E_{b}})\oplus V^0_{-\infty} (\R)
\]
satisfies the following properties:
\begin{itemize}
\item[(i)] The linear map $L^0_{b, \R}$ is an injection for any $b \in B$. 
\item[(ii)] For any $b \in B$, there is a natural isomorphism 
\[
\operatorname{Coker} L^0_{b, \R}\cong H^+(X_b; \R), 
\]
where $H^+(X; \R)$ is the space of self-dual harmonic $2$-forms on $X$.
Moreover, the correspondence 
\[
 b\mapsto  \operatorname{Coker} L^0_{b, \R} \cong H^+(X_b; \R)
 \]
 gives a section of 
\[
\mathrm{Gr}^{+}(H^{2}(X;\R)) \to \mathrm{Gr}^{+}_{E} \to B. 
\]
\end{itemize}
\end{lem}
\begin{proof}
 In order to prove (i), we consider the following two operators $L^{AHS}_b$ and $\wt{L}_b^{AHS}$.
 
The first operator is the Atiyah--Hitchin--Singer operator with a spectral boundary condition
\[
L^{AHS}_b:= d^*+d^+ + \pr_{H^-}\circ \wt{r} :   L^2_k ( i\Lambda^1_{E_b}) \to L^2_{k-1} ( i\Lambda^0_{E_b} \oplus i \Lambda^+_{E_b}) \oplus H^-,  
\]
 where 
\begin{itemize}
\item The linear space $H^-$ is the $L^2_{k-\frac{1}{2}}$-completion of the non-positive eigenspace of the operator 
\[
\wt{l}  :i \im d \oplus i\Ker d^*  \oplus   i\Om^0(Y)\to i\im d \oplus i\Ker d^*  \oplus  i \Om^0(Y)
 \]
 defined by 
 \[
 \wt{l}  :=  \begin{pmatrix} 
 0 & 0 & -d \\
 0 & * d & 0  \\
 -d^* & 0 &0 \\
  \end{pmatrix}. 
  \]
 \item The operator 
$\wt{r}_b : L^2_k ( i\Lambda^1_{E_b}) \to  V$ comes from the restriction with respect to the inclusion $i_b: Y= \partial E_b \to E_b$, where 
\[
V:= d(L^2_{k-\frac{1}{2}}( i\Lambda^0(Y))) \oplus i\Ker (d^*|_{ L^2_{k-\frac{1}{2}}( i\Lambda^1(Y))}   ) \oplus   L^2_{k-\frac{1}{2}}( i\Lambda^0(Y)) . 
 \]
The map $\wt{r}_b $ is defined by 
 \[
 \wt{r}_b (a ) := ( \pr_{d(L^2_{k-\frac{1}{2}}( i\Lambda^0(Y)))} i_b^*a , \pr_{ i\Ker (d^*|_{ L^2_{k-\frac{1}{2}}( i\Lambda^1(Y))} ) } i_b^*a,  i_b^* \iota_{\partial_t}  a_0 ) . 
 \]
 \item The operator $\pr_{H^-} : V \to H^- $ is the $L^2$-projection.
 \end{itemize}
 Regarding the first operator $L^{AHS}_b$, it is proved in \cite[Proposition 3.11]{APSI} that there exist isomorphisms 
 \begin{align}
 \label{APS result}
 \begin{split}
 \Ker L^{AHS}_b \cong \Set{ a \in L^2(\Lambda^1_{E_b^*}) | d^*a=d^+a=0 } \text{ and }\\ 
 \Coker L^{AHS}_b \cong \Set{ (b_1, b_2)  \in L^2_{\text{ex}}(\Lambda^0_{E_b^*}\oplus \Lambda^+_{E_b^*} ) |  d b_1 = 0 \text{ and } d^* b_2 = 0  } .
 \end{split}
 \end{align} 
 for each $b \in B$, where $E_b^* := E_b \cup [1,\infty) \times Y$ and $L^2_{\text{ex}}$ means extended $L^2$-sections introduced in \cite[page 58]{APSI}. Note that $E^*:= \bigcup_{b \in B} E_b^* \to B$ gives a fiber bundle whose fiber is the non-compact $4$-manifold $E_b^*$.
 In \eqref{APS result}, we have used a fiberwise Riemann metric $g_{E^*}$ on $E^*$ defined as an extension of $g_E$ such that 
 \[
 g_{E^*}|_{[1, \infty)\times Y } = \pi^* g_Y + dt^2 . 
 \]
  By integration by parts, we can conclude that 
  \begin{align}\label{APS2} 
 \Set{ a \in L^2(\Lambda^1_{E_b^*}) | d^*a=d^+a=0 } = \Set{ a \in L^2(\Lambda^1_{E_b^*}) | d^*a=da=0 }. 
 \end{align}
 Moreover, \cite[Proposition 4.9]{APSI} implies that  
\begin{align}\label{APS3}
\Set{ a \in L^2(\Lambda^1_{E_b^*}) | d^*a=da=0 }= \im ( H^1(E_b,\partial E_b; \R) \to H^1(E_b;\R) ) = \{0\}.  
\end{align}
Combining \eqref{APS result}, \eqref{APS2} and \eqref{APS3}, we obtain
\[
 \Ker L^{AHS}_b = \{0\}. 
 \]
Recall that we imposed $b_1(Y)=0$, and the kernel of $\wt{l}$ consists of constant functions on $Y$. This implies that, for any element $(b_1, b_2)  \in L^2_{\text{ex}}(\Lambda^0_{E_b^*}\oplus \Lambda^+_{E_b^*})$, there exists a constant $c$ such that $(b_1-c, b_2) \in L^2(\Lambda^0_{E_b^*}\oplus \Lambda^+_{E_b^*})$. By combining this observation with \cite[Proposition 4.9]{APSI} and \cite[Corollary 4.11]{APSI}, we can see that 
\begin{align} \label{APS4}
 \Set{ (b_1, b_2)  \in L^2_{\text{ex}}(\Lambda^0_{E_b^*}\oplus \Lambda^+_{E_b^*} ) |  d b_1 = 0 \text{ and } d^* b_2 = 0  }  \cong  H^0 (E_b; \R) \oplus H^+(E_b; \R)
\end{align}
for each $b \in B$.
As a conclusion, we obtain 
\begin{align} \label{APS5}
 \Coker L^{AHS}_b \cong H^0 (E_b; \R) \oplus H^+(E_b; \R). 
\end{align}
Since isomorphisms \eqref{APS result} and \eqref{APS4} vary continuously with respect to $b \in B$, the isomorphism \eqref{APS5} gives a fiberwise isomorphism.

The second operator $\wt{L}^{AHS}_b$ is the AHS operator with a projection
\begin{align*}
\wt{L}^{AHS}_b &:= d^*+d^+ + (\pr_{H^-} + \Pi)\circ \wt{r}\\
&\quad:   L^2_k ( i\Lambda^1_{E_b}) \to L^2_{k-1} ( i\Lambda^0_{E_b} \oplus i \Lambda^+_{E_b}) \oplus V^0_{-\infty}(\R) \oplus \underline{W}_Y, 
\end{align*}
where 
$\underline{W}_Y =\mathcal{H}^0_Y \oplus d(L^2_{k-\frac{1}{2}}( i\Lambda^0(Y)))$, and
the map $ \Pi$ is the $L^2$-projection
\[
\Pi :V \to \mathcal{H}^0_Y \oplus d(L^2_{k-\frac{1}{2}}( i\Lambda^0(Y))) =\underline{W}_Y. 
\]
Here $\mathcal{H}^0_Y$ is the space of $i\R$-valued constant functions on $Y$.

 We shall compare $L_b^{AHS}$ with $L^0_b$ via $\wt{L}_b^{AHS}$.
First let us compare $L_b^{AHS}$ with $\wt{L}_b^{AHS}$:

\begin{lem}
The kernels and cokernels of $\wt{L}^{AHS}_b$ and ${L}^{AHS}_b$ are isomorphic to each other respectively, via the following isomorphism between the codomains of $\wt{L}^{AHS}_b$ and ${L}^{AHS}_b$:
\begin{align*}
\id \oplus \Pi : L^2_{k-1} ( i\Lambda^0_{E_b} \oplus i \Lambda^+_{E_b}) \oplus H^- 
 \to L^2_{k-1} ( i\Lambda^0_{E_b} \oplus i \Lambda^+_{E_b}) \oplus V^0_{-\infty}(\R) \oplus \underline{W}_Y,
\end{align*}
which is defined by 
\[
\id \oplus \Pi ( x_1, x_2 , (y_1, y_2, y_3)  ) := (x_1, x_2, y_2, \Pi (y_1, y_2, y_3 )) . 
\]

\end{lem}
\begin{proof}
The operator $\wt{l}$ can be written as the sum  of $*d$ on $\Ker d^*$ and $ l$, where $l$ is the self-adjoint operator 
 \[
  l:=   \begin{pmatrix} 
 0 & -d^*\\
 -d & 0 \\
  \end{pmatrix} :i \im d  \oplus   i\Om^0(Y)\to i\im d \oplus  i \Om^0(Y). 
  \]
  Let us denote by $\wt{H}^-$ the non-positive eigenspace of $l$.
It is checked in \cite{Kha15} that both of $\wt{H}^-$ and $\underline{W}_Y$ have $L^2_{k-\frac{1}{2}} (i\Lambda^0 (Y))_0\oplus 0$ as a complement in 
\[
L^2_{k-\frac{1}{2}} (i\Lambda^0 (Y)) \oplus dL^2_{k+\frac{1}{2}} (i\Lambda^0 (Y)),
\] 
where 
\[
L^2_{k-\frac{1}{2}} (i\Lambda^0 (Y))_0:= \Set{ a\in L^2_{k-\frac{1}{2}} (i\Lambda^0 (Y)) | \int_Y a \operatorname{dvol} = 0}. 
\]
 This proves $\id \oplus \Pi$ is an isomorphism. 
\end{proof}
Next, we compare $\wt{L}_b^{AHS}$ with $L^0_b$.
We have the following commutative diagram: 
\[
  \begin{CD}
     0 @. 0 \\
  @VVV    @VVV \\
  L^2_{k} ( i\Lambda^1)_{CC}   @>{ L^0_{b, \R} }>>   L^2_{k-1} ( i\Lambda^+)\oplus V^\mu_{-\infty} (\R)   \\
    @VVV    @VVV \\
      L^2_k ( i\Lambda^1_{E_b})  @>{\wt{L}^{AHS}_b}>>  L^2_{k-1} ( i\Lambda^0_{E_b} \oplus i \Lambda^+_{E_b}) \oplus V^0_{-\infty}(\R) \oplus \underline{W}_Y \\ 
      @V{d^* \oplus \prod \circ r} VV    @VVV \\
       L^2_{k-1} ( i\Lambda^0_{E_b})_0 \oplus  \underline{W}_Y @>>>  L^2_{k-1} ( i\Lambda^0_{E_b}) \oplus  \underline{W}_Y \\
         @VVV    @VVV \\
              0 @. 0,   \\
  \end{CD}
\]
where 
\[
  L^2_{k-1} ( i\Lambda^0_{E_b})_0 := \Set{ a \in L^2_{k-1} ( i\Lambda^0_{E_b}) |  \int_X a \operatorname{dvol} = 0}. 
  \]
It follows from this diagram and the snake lemma that there are fiberwise isomorphisms 
\[
\Ker L^0_{b} |_{ L^2_{k} ( i\Lambda^1)_{CC}  } \cong  \Ker \wt{L}^{AHS}_b
\]
and 
\[
\Coker L^0_{b} |_{ L^2_{k} ( i\Lambda^1)_{CC}  } \oplus H^0 (E_b; \R)   \cong \Coker \wt{L}^{AHS}_b. 
\]
By combining this with \eqref{APS result}, we conclude that there are fiberwise isomorphisms
 \begin{align*}
\Ker L^0_{b, \R} \cong \{0\},\quad \Coker L^0 _{b, \R} \cong H^+ (E_b; \R). 
 \end{align*}
This completes the proof of \cref{linear injection}.
\end{proof}


Next, to carry out finite-dimensional approximation, we take a sequence of finite-dimensional vector subbundles $W_1^n$ of  $L^2_{k-1} ( i\Lambda^+_{E})\oplus L^2_{k-1} ( S^-_E)$. 

\begin{lem} \label{fin dim app1} There exists a sequence of finite-dimensional vector subbundles $W_1^n$ of  $L^2_{k-1} ( i\Lambda^+_{E})\oplus L^2_{k-1} ( S^-_E)$
 such that 
 \begin{itemize}
 \item the sequence is an increasing sequence 
 \[
 W_1^0 \subset W_1^1 \subset W_1^2 \subset W_1^3 \subset \cdots \subset L^2_{k-1} ( i\Lambda^+_{E})\oplus L^2_{k-1} ( S^-_E), 
 \]
 \item the equality 
\begin{align}\label{trans}
\operatorname{Im} L^0_{b} \cap (L^2_{k-1} ( i\Lambda^+_{E_b})\oplus L^2_{k-1} ( S^-_{E_b}) +   (W_1^0)_{b} ) = L^2_{k-1} ( i\Lambda^+_{E_b})\oplus L^2_{k-1} ( S^-_{E_b})
\end{align}
holds for all $b \in B$, and
\item  the projection $\pr_{(W_1^n)_b } : L^2_{k-1} ( i\Lambda^+_{E_b})\oplus L^2_{k-1} ( S^-_{E_b}) \to (W_1^n)_b$ satisfies 
\[
\| \pr_{(W_1^n)_b } \gamma_b  - \gamma_b \|_{L^2_{k-1}} \to 0 \text{ as } n \to \infty
\]
 for any $\gamma_b \in  L^2_{k-1} ( i\Lambda^+_{E_b})\oplus L^2_{k-1} ( S^-_{E_b})$  and $b\in B$.
\end{itemize}
\end{lem}
\begin{proof}
For a fixed point $b_1 \in B$, we define 
\[
W_{b_1} := (\im L^{\mu_n}_{b_1} \cap L^2_{k-1} ( i\Lambda^+_{E})\oplus L^2_{k-1} ( S^-_E) )^{\perp_{L^2_{k-1}} } .
\]
 By using a global trivialization of $L^2_{k-1} ( i\Lambda^+_{E})\oplus L^2_{k-1} ( S^-_E)$, we extend a vector space $W_{b_1}$ to a subbundle $\wt{W}_{b_1}$ of $L^2_{k-1} ( i\Lambda^+_{E})\oplus L^2_{k-1} ( S^-_E) \to B$. 
Since surjectivity is an open condition, for any element $b$ in a small neighborhood of $b_1$, 
\[
\operatorname{Im} L^0_{b} +    \wt{W}_{b_1} = L^2_{k-1} ( i\Lambda^+_{E})\oplus L^2_{k-1} ( S^-_E). 
\]
 Since $B$ is compact, we can take a finite sequence of points $ b_1, \cdots, b_k$ of $B$ and a finite sequence of subbundles $\wt{W}_{b_1}, \cdots , \wt{W}_{b_k}$ of $L^2_{k-1} ( i\Lambda^+_{E})\oplus L^2_{k-1} ( S^-_E) \to B$ such that for any $b \in B$, we have that
\[
\operatorname{Im} L^0_{b} \oplus    \bigoplus_{  1 \leq i \leq k}  \wt{W}_{b_i}|_b = L^2_{k-1} ( i\Lambda^+_{E_b})\oplus L^2_{k-1} ( S^-_{E_b}).
\]
Define 
\[
W_1^1 := \bigoplus_{  1 \leq i \leq k}  \wt{W}_{b_i}. 
\]
By using a global trivialization of the Hilbert bundle 
\[
L^2_{k-1} ( i\Lambda^+_{E})\oplus L^2_{k-1} ( S^-_E) \cap (W_1^1)^{\perp_{L^2_{k-1}}} \to B, 
\]
we set 
\[
W_1^n := W_1^1 \oplus \R\langle e_1 \rangle \oplus \cdots \oplus \R\langle e_{n-1} \rangle, 
\]
where $\{e_i\}_{ i \in \Z_{>0}} $ is a global orthonormal basis of $L^2_{k-1} ( i\Lambda^+_{E})\oplus L^2_{k-1} ( S^-_E) \cap (W_1^1)^{\perp_{L^2_{k-1}}}$. 
Then one can check the third condition.
\end{proof}

Take sequences of numbers $\lambda_n$ and $\mu_n$ such that $\lambda_n \to -\infty$ and $\mu_n \to \infty$ as $n \to \infty$. 
By \cref{fin dim app1}, we take a sequence of finite-dimensional vector subbundles $W_1^n$ of  $L^2_{k-1} ( i\Lambda^+_{E})\oplus L^2_{k-1} ( S^-_E)$ satisfying the conclusions of \cref{fin dim app1}. 
Let us define 
\[
W_0^n:= ( {L}^{\mu_n} )^{-1} ( W_1^n \oplus \underline{V_{\lambda_n}^{\mu_n}}  ) .
\]
 By \eqref{trans}, we can see that $W_0^n$ are finite dimensional subbundles of $ L^2_{k} ( i\Lambda^1)_{CC}  \oplus L^2_k ( S^+_E)$. The following lemma tells us the injectivity of $L^\mu_b$ for a sufficiently large $\mu$: 
\begin{lem} \label{inj} There exists $\mu_0>0$ such that, for any $\mu$ with $\mu>\mu_0$ and for any $b \in B$, 
 $L^\mu_b$ is injective. 
 \end{lem}
 \begin{proof}
 Suppose that the conclusion is not true. Then we have sequences of points $\{\mu_n\}$ and $\{b_n \} \subset B$ and $x_{b_n} \in  L^2_{k} ( i\Lambda^1_{E_{b_n}})_{CC}$ such that 
 $\mu_n \to \infty$, $L^\mu_{b_n} (x_{b_n} ) =0$ and $x_{b_n}\neq 0$.  By scalar multiplication, we may assume that $\|x_{b_n}\|^2_{L^2_k}=1$. Since $B$ is compact, after taking a subsequence, we can assume that $\{b_n\}$ converges to some point $b_\infty \in B$.
 By the Fredholm property of ${L}^{\lambda_n}$, after taking a subsequence, we can assume that $\{x_{b_n} \}$ converges to some point $x_{b_\infty}$ which satisfies 
 \[
d^* (x_{b_\infty}) =0, ~ d^+ (x_{b_\infty}) =0, ~  r_{b_\infty } x_{b_\infty} =0, ~ \|x_{b_\infty}\|^2_{L^2} =1  \text{ and } D^+_{\widehat{A}_{b_\infty}}( x_{b_\infty}) =0. 
\] 
However this contradicts the unique continuation property of the operator $( d^* , d^+  , D^+_{\widehat{A}_{b_\infty}})$. 
 \end{proof}
 
We have an isomorphism 
\[
W_1^n + \underline{V^{\mu_n}_{\lambda_n}} + \operatorname{Ker} L^{\mu_n} \cong W_0^n + \operatorname{Coker} L^{\mu_n}
\]
between the virtual vector bundles.
As it is mentioned in \cite[page 923]{Ma03}, 
\[
\operatorname{Coker} L^{\mu_n} \cong \operatorname{Coker} L^{0}  \oplus  \operatorname{Coker} p_{0}^{\mu_n} \circ  \pr_{\Ker d^{\ast}}. 
\]
Moreover, if $n$ is sufficiently large, since 
\[
p_{0}^{\mu_n} \circ  \pr_{\Ker d^{\ast}}  \circ i^{\ast} : \Ker L^{0} \to V_{0}^{\mu_n}
\]
 is fiberwise injective by \cref{fin dim app1},
we have an identification
\[
\Ker L^0-  \operatorname{Coker} L^0 + \operatorname{Coker} L^{\mu_n} \cong \underline{V^{\mu_n}_0},
\]
and thus have
\begin{align}\label{decomp}
W_1^n + \underline{V^{\mu_n}_{\lambda_n}} + \Ker L^0- \operatorname{Coker} L^0 \cong   W^n_0 + \underline{V^{\mu_n}_0}
\end{align}
as virtual vector bundles over $B$.

Applying the projection, we obtain a family of maps
\[
\pr_{W_1^n\times \underline{V_{\lambda_n}^{\mu_n}}}  \circ \mathcal{F}^{\mu_{n}} |_{W_0^n} :W_0^n \to W_1^n \times  \underline{V_{\lambda_n}^{\mu_n}}
\]
whose $S^1$-invariant part is given by 
\[
(\pr_{W_1^n\times \underline{V_{\lambda_n}^{\mu_n}}}   \circ \mathcal{F}^{\mu_{n}} |_{W_0^n})^{S^1}  :W_0^n(\R)  \to W_1^n(\R) \times  \underline{V_{\lambda_n}^{\mu_n} (\R)}.
\]
This induces a map 
\begin{align}\label{fin app}
\pr_{W_1^n\times \underline{V_{\lambda_n}^{\mu_n}}}  \circ \mathcal{F}^{\mu_{n}} |_{W_0^n} :B(R,W_0^n) \to (W_1^n \times  \underline{V_{\lambda_n}^{\mu_n}} )^{+_B} , 
\end{align}
where $+_B$ denotes the fiberwise one-point compactification. 

For a subset $A$ in $V_{\lambda}^{\mu}$, set
\begin{align*}
&A^{+} := \Set{x \in A | \forall t>0,\  t \cdot x \in A}.
\end{align*}
To obtain a suitable index pair used for a Bauer--Furuta-type invariant from \eqref{fin app}, we need the following \cref{lem: ensure the existence of Conley index}.
Set
\begin{align}
\label{eq: tilde K one}
\wt{K}_{1 }  := B(R,W_0^n) \cap ( (\pr_{W_1^n}  \circ \mathcal{F}^{\mu_{n}})^{-1}    B(\epsilon_n, W_1^n) ) 
\end{align}
and 
\begin{align}
\label{eq: tilde K two}
\wt{K}_{2} := S(R,W_0^n) \cap ( (\pr_{W_1^n }  \circ \mathcal{F}^{\mu_{n}})^{-1}    B(\epsilon_n, W_1^n) )  
\end{align}
for a sequence of positive real numbers $\{ \epsilon_n \}_{ n \in \Z_{>0}} $ with $ \epsilon_n \to 0$ as $n\to \infty$.

\begin{lem}
\label{lem: ensure the existence of Conley index}
Suppose that the base space $B$ is compact. 
For sufficiently large $R, R'$ and $n$, the compact sets
\[
K_{1 }  := p_{V_{\lambda_n}^{\mu_n} }  \circ \mathcal{F}^{\mu_{n}}(\tilde{K}_{1})
\]
and 
\[
K_{2} :=  p_{V_{\lambda_n}^{\mu_n} }  \circ \mathcal{F}^{\mu_{n}}(\tilde{K}_{2})
\]
satisfy the assumption of \cite[Theorem~4]{Ma03}, \cite[Lemma~A.4]{Kha15} for $A:=B(R' ; V^{\mu_n}_{\lambda_n} )$, i.e. the following conditions hold: 
\begin{itemize}
\item[(i)]  if $x \in K_{1} \cap A^+ $, then $([0,\infty) \cdot x) \cap \partial A=\emptyset$, and 
\item[(ii)] $K_{2} \cap A^+  = \emptyset $.
\end{itemize}
\end{lem}

\begin{proof}
The proof is essentially the same as the proof of \cite[Proposition~4.5]{Kha15}.
We will prove by contradiction. 
First let us verify (i).
Before starting the discussion, we fix universal constants $B_k$ and $C_k^0$ of \cite[Corollary~4.3]{Kha15} for a family of metrics $g_{E_b}$, i.e. we can take constants $B_k$ and $C_k$ satisfying the following conditions: for any finite $X$-trajectory with respect to a metric $g_{E_b}$ induced from a pair $(x, y)$ of solution $x  \in  L^2_{k} ( i\Lambda^1_{E_b} )_{CC}  \oplus L^2_k ( S^+_{E_b} )$ for some $b \in B$ and a half trajectory $y : [0, \infty) \to \operatorname{Coul}_k(Y,\frakt)$ satisfying 
\[
\frac{\partial}{\partial t} y (t) = - (l +  c) (y(t))  \text{ and } y(0) =  r_{b}  x, 
\]
we have
\begin{itemize}
\item $\| x\|_{L^2_{k}} \leq B_k$ and 
\item for each $t \in [0,\infty)$, $\|y(t) \|_{L^2_{k-\frac{1}{2}}} \leq C_k$.
\end{itemize}
For the definition of finite $X$-trajectory, see \cite[Subsection 4.1]{Kha15}.
The existences of such constants follow since $B$ is compact.
Fix constants $R$ and $R'$  with $ R' >C_k$ and $R>B^0_k$ for a fixed $k$. 

We suppose that there exist a sequence $\{b_n\}_{n\in \Z_{>0} } \subset B$ and a sequence $\{x_{n}\}_{n\in \Z_{>0}}$ satisfying that
\[
x_n \in B(R,W_0^n|_{b_n}) \cap ( (\pr_{W_1^n }  \circ \mathcal{F}^{\mu_{n}})^{-1}    B(\epsilon_n, W_1^n|_{b_n}) )
\]
and that there exists a sequence of approximated half trajectories $y_n : [0,\infty) \to V_{\lambda_n}^{\mu_n} $ with 
\[
\frac{\partial}{\partial t} y_n (t) = - (l + p_\lambda^\mu c) (y(t)) , ~ y_n(0) =p^{\mu_n}_{-\infty}  r_{b_n}  x_n
\]
and 
\[
 \| y_n(t_n) \|_{V^{\mu_n}_{\lambda_n} } = R'. 
 \]
After taking a subsequence, we suppose $b_n \to b_\infty$ as $n\to \infty$.
We need the following lemma to get a contradiction: 
\begin{lem}\label{converges}
 Let $\{ x_n\}_{  n\in \Z_{>0}} $ be a bounded sequence in $L^2_{k} ( i\Lambda^1)_{CC}  \oplus L^2_k (S^+_{E} ) $ such that 
\[
L^{\mu_n}_{b_n} x_n \in    W_1^n ,  p^{\mu_n}_{-\infty} r_{b_n} x_n\in   V_{\lambda_n}^{\mu_n}  , \text{ and }(L^{\mu_n}_{b_n}  + p^{\mu_n}_{-\infty}  C_{b_n}  ) (x_n) \to 0 
\]
 in $L^2_k$-norm, where $b_n$ is the corresponding base point of $x_n$ in $B$.
  We also suppose that there exists a sequence of approximated half trajectories $y_n : [0,\infty) \to V_{\lambda_n}^{\mu_n} $ satisfying
\[
\frac{\partial}{\partial t} y_n (t) = - (l + p_\lambda^\mu c) (y(t))  \text{ and } y_n(0) =p^{\mu_n}_{-\infty} r_{b_n}  x_n. 
\]
Then, after taking a subsequence, the sequence $\{ b_n \}$ converges a point $b_\infty \in B$, the sequence $\{x_n\}$ converges to a solution $x_\infty$ to the Seiberg--Witten equations for $E_{b_\infty}$ and there exists a Seiberg--Witten half trajectory $y_\infty$ satisfying $\frac{\partial}{\partial t} y (t) = - (l +  c) (y(t)) $,  $y_\infty(0)=r_{b_\infty}  x_\infty$ and $y_n(t)$ converges $y_\infty(t)$ for all $t$ in $L^2_{k- \frac{1}{2}}$. 

\end{lem}
This is a family version of \cite[Lemma~4.4]{Kha15}, and we omit the proof since that is essentially the same as the proof of \cite[Lemma~4.4]{Kha15}. (Here we also use the compactness of $B$. ) 
By \cref{converges}, we have $x_\infty$ and $y_\infty$ satisfying the conclusion of \cref{converges}. 
After taking a subsequence of $\{t_n\}$, we have two cases: $t_n \to t_\infty \in [0, \infty)$, or $t_n \to \infty$. 
This implies that
\[
\|y_\infty(t_\infty)\|_{L^2_{k-\frac{1}{2}}} =R' \text{ or }\|y_\infty(\infty)\|_{L^2_{k-\frac{1}{2}}} =R'
\]
 holds. 
However,  this contradicts the choice of $R'>B_k$. 

Next, we prove the case (ii). 
We suppose that there exist a sequence $\{b_n\}_{n\in \Z_{>0} } \subset B$, a sequence 
\[
x_n \in  S(R,W_0^n) \cap ( (\pr_{W_1^n }  \circ \mathcal{F}^{\mu_{n}})^{-1}    B(\epsilon_n, W_1^n) ) 
\]
and a sequence of approximated half trajectories $y_n : [0,\infty) \to V_{\lambda_n}^{\mu_n} $ with 
\[
\frac{\partial}{\partial t} y_n (t) = - (l + p_\lambda^\mu c) (y(t)) , ~ y_n(0) =p^{\mu_n}_{-\infty}  r_{b_n}  x_n. 
\]
Again, we apply \cref{converges} and obtain $x_\infty$ and $y_\infty$ satisfying the conclusion of \cref{converges}. 
Note that we have 
\[
\| x_\infty \|_{L^2_k} = R. 
\]
This contradicts the choice of $R$.
\end{proof}

By using the above lemma and \cite[Theorem~4]{Ma03}, we may take an $S^1$-invariant Conley index $(N_n, L_n)$ such that 
\[
(K_{1 }, K_2 ) \subset (N_n, L_n) . 
\]
Then $\pr_{W_1^{n}}  \circ  \mathcal{F}^{\mu_{n}}|_{W_{0}^{n}}$ induces an $S^1$-equivariant continuous map
\begin{align} \label{simplify}
f_n :  (W_0^n(R))^{+_B}   \to  (W_1^n / (W_1^n \setminus B(\epsilon_n, W_1^n ))) \wedge_B \underline{I_{\lambda_n}^{\mu_n} }
 \end{align}
 as in \cite[Section~9]{Ma03},
 where $\wedge_B$ denotes the fiberwise smash product. 
We call this map \eqref{simplify} the {\it families relative Bauer--Furuta invariant}.

 The decomposition \eqref{decomp} implies that this map stably can be written so that 
 \begin{align*}
 \begin{split}
f :  \left(\Set{\ind D^{+}_{\widehat{A}_{b}}}_{b \in B}\right)^{+_{B}}
 \to \left(\underline{\C}^{n(Y , \frakt, g)} \oplus \underline{\R}^{b^+(X)} \right)^{+_{B}} \wedge_{B} \underline{SWF(Y, \frakt)},
  \end{split}
 \end{align*}
 where $\{\ind D^{+}_{\widehat{A}_{b}}\}_{b \in B}$ denotes the virtual index bundle.
  Arguing exactly as in \cite{Ma03}, one may see that this map gives rise to a topological invariant of a smooth bundle $E$ of $4$-manifolds with boundary equipped with a fiberwise spin$^{c}$ structure,
  but the invariance is not necessary for our purpose in this paper.
  
When $\mathfrak{s}$ is spin, respecting $\Pin(2)$-symmetry over the whole argument above, we obtain a $\Pin(2)$-equivariant map
 \begin{align*}
  \begin{split}
f :  \left(\Set{\ind D^{+}_{\widehat{A}_{b}}}_{b \in B}\right)^{+_{B}}
 \to \left(\underline{\quat}^{n(Y , \frakt, g)/2} \oplus \underline{\tilR}^{b^+(X)} \right)^{+_{B}} \wedge_{B} \underline{SWF(Y, \frakt)}
  \end{split}
 \end{align*}
as well.

\section{Proof of the main theorems}
\label{section: Proof of the main theorem}

In this \lcnamecref{section: Proof of the main theorem} we give the proofs of the main theorems, \cref{theo: main theo,theo: main theo2}.

\subsection{Properties of the families relative Bauer--Furuta invariant}
\label{subsection: Properties of the relative Bauer--Furuta invariant}

In this \lcnamecref{subsection: Properties of the relative Bauer--Furuta invariant} we summarize some properties of the relative families Bauer--Furuta invariant~\eqref{simplify} which are deduced from \cref{subsection: families relative Bauer--Furuta invariant}.
Henceforth we shall drop $n$ in \eqref{simplify} from our notation.
Recall that the families relative Bauer--Furuta invariant for the smooth family $(X,\fraks) \to E \to B$ is given as a fiberwise $S^{1}$-equivariant map between families of pointed $S^{1}$-spaces parametrized over $B$:
\begin{align}
\label{eq: starting point}
f : W_{0}^{+_{B}} 
\to W_{1}^{+_{B}} \wedge_{B} \underline{I_{\lambda}^{\mu}}.
\end{align}
Here 
\begin{itemize}
\item  $I_{\lambda}^{\mu}$ is the Conley index used to define the Seiberg--Witten Floer homotopy type of $Y$, where $\mu, -\lambda$ are taken to be sufficiently large.
Let $(N,L)$ be an index pair to define $I_{\lambda}^{\mu}$ given in \eqref{Conley index} so that
\[
I_{\lambda}^{\mu} = N/L. 
\]
\item $W_{0}, W_{1} \to B$ are vector bundles.
Each $W_{i}$ is the direct sum of a real vector bundle $W_{i}(\R)$ and a complex vector bundle $W_{i}(\C)$ over $B$:
\[
W_{i} = W_{i}(\R) \oplus W_{i}(\C).
\]
The $S^1$-actions on $W_{0}$ and $W_{1}$ are given as trivial and the action coming from the structure of a $\C$-vector space. 
\item The $S^{1}$-invariant part of the map \eqref{eq: starting point} is obtained as the restriction of a fiberwise $S^{1}$-equivariant linear map between vector bundles, denoted also by the same symbol $f^{S^{1}}$ by an abuse of notation:
\begin{align*}
f^{S^{1}} : W_{0}(\R) \to W_{1}(\R) \oplus \underline{V_{\lambda}^{\mu}(\R)} = W_{1}(\R) \times V_{\lambda}^{\mu}(\R).
\end{align*}
Let $p_{V_{\lambda}^{0}(\R)} : V_{\lambda}^{\mu}(\R) \to V_{\lambda}^{0}(\R)$ be the $L^{2}$-projection.
It follows from \cref{linear injection} that the map
\begin{align}
\label{eq: fix f 2}
(\id_{W_{1}(\R)} \oplus p_{V_{\lambda}^{0}(\R)}) \circ f^{S^{1}} : W_{0}(\R) \to W_{1}(\R) \oplus \underline{V_{\lambda}^{0}(\R)} = W_{1}(\R) \times V_{\lambda}^{0}(\R)
\end{align}
is a fiberwise linear injection and its fiberwise cokernel is isomorphic to the bundle $H^{+}(E) \to B$.
\item We have
\begin{align}
\begin{split}
\label{eq: dirac index formula}
\rank_{\C}W_{0}(\C) - \rank_{\C}W_{1}(\C) 
&= \ind_{\C}{D^{+}_{\hat{A}_{b}}} + \dim_{\C} V^{0}_{\lambda}(\C)\\
&= \frac{c_{1}(\fraks)^{2} - \sigma(X)}{8} + n(Y,\frakt,g)+ \dim_{\C} V^{0}_{\lambda}(\C).
\end{split}
\end{align}
Here $\{\hat{A}_{b}\}_{b \in B}$ denotes a family of $U(1)$-connections of the family of the determinant line bundles and $\{\ind{D^{+}_{\hat{A_{b}}}}\}_{b \in B}$ denotes the index of the families of the Dirac operators associated to $E$.
\end{itemize}

To prove \cref{theo: main theo}, we have to rewrite the $S^{1}$-fixed part $(I_{\lambda}^{\mu})^{S^{1}}$ into the sphere $V_{\lambda}^{0}(\R)^{+}$ without loss of information about the image of $f^{S^{1}}$.
It is summarized as the following \cref{lem: htpy NL}.
Let 
\begin{align*}
&p_{1} : W_{1} \times V_{\lambda}^{\mu} \to W_{1},\\
&p_{2} : W_{1} \times V_{\lambda}^{\mu} \to V_{\lambda}^{\mu}
\end{align*}
be the projections.

\begin{lem}
\label{lem: htpy NL}
There exists a homotopy equivalence
\[
\varphi : N^{S^{1}}/L^{S^{1}} \to V_{\lambda}^{0}(\R)^{+}
\]
for which the diagram 
\begin{align}
\begin{split}
\label{eq: htpy diagram Tira}
\xymatrix{
W_{0}(\R)^{+_{B}} \ar[r]^-{p_{2} \circ f^{S^{1}}} \ar[rd]_{p_{V_{\lambda}^{0}(\R)} \circ f^{S^{1}}\quad}
&N^{S^{1}}/L^{S^{1}} \ar[d]^-{\varphi} \\
&V_{\lambda}^{0}(\R)^{+}
}
\end{split}
\end{align}
commutes up to homotopy.
\end{lem}

\begin{proof}
Since $p_{1} \circ f^{S^{1}} : W_{0}(\R) \to W_{1}(\R)$ is a fiberwise linear map,
\[
\tilde{D}(W_{0}(\R)) := D(W_{0}(\R)) \cap \tilde{K_{1}} = D(W_{0}(\R)) \cap (p_{1} \circ f^{S^{1}})^{-1}(B(\epsilon;W_{1}(\R)))
\]
 and 
 \[
\tilde{S}(W_{0}(\R)) := S(W_{0}(\R)) \cap \tilde{K_{1}} = S(W_{0}(\R)) \cap (p_{1} \circ f^{S^{1}})^{-1}(B(\epsilon;W_{1}(\R)))
\]
are a disk bundle and a sphere bundle of $W_{0}(\R)$ of some common radius respectively.
Here $\tilde{K}_{1}, \tilde{K}_{2}$ are defined in \eqref{eq: tilde K one} and \eqref{eq: tilde K two}.

Let us remark that we have
\begin{align}
\label{eq: inc pair1}
(p_{2} \circ f(\tilde{D}(W_{0}(\R))), p_{2} \circ f(\tilde{S}(W_{0}(\R))))
\subset (K_{1}^{S^{1}},K_{2}^{S^{1}}) \subset (N^{S^{1}},L^{S^{1}}).
\end{align}
On the other hand, since the map \eqref{eq: fix f 2} is a fiberwise linear injection, we have also that
\begin{align}
\begin{split}
\label{eq: inc pair2}
&(p_{2} \circ f(\tilde{D}(W_{0}(\R))), p_{2} \circ f(\tilde{S}(W_{0}(\R))))\\
\subset& (D(V_{\lambda}^{0}(\R)) \times D(V_{0}^{\mu}(\R)), S(V_{\lambda}^{0}(\R)) \times D(V_{0}^{\mu}(\R))),
\end{split}
\end{align}
where $D(\cdot)$ and $S(\cdot)$ are disks and spheres with appropriate radius respectively.
Moreover, it is easy to check that both of the right-hand sides of \eqref{eq: inc pair1} and \eqref{eq: inc pair2} are index pairs for the $S^{1}$-invariant part of the isolated invariant set $\mathrm{Inv}(B(2R; V_{\lambda}^{\mu}(\R)))$.

It follows from this combined with an argument used to prove Proposition~A.5 \cite{Kha15} by Khandhawit that there exists a homotopy equivalence
\[
\phi : N^{S^{1}}/L^{S^{1}} \to \frac{D(V_{\lambda}^{0}(\R)) \times D(V_{0}^{\mu}(\R))}{S(V_{\lambda}^{0}(\R)) \times D(V_{0}^{\mu}(\R))}
\]
which makes the diagram
\begin{align}
\begin{split}
\label{eq: htpy diagram1 Tira}
\xymatrix{
\tilde{D}(W_{0}(\R))/\tilde{S}(W_{0}(\R)) \ar[r]^-{p_{2} \circ f^{S^{1}}} \ar[rd]_{p_{2} \circ f^{S^{1}}}
&N^{S^{1}}/L^{S^{1}} \ar[d]^-{\phi} \\
&\frac{D(V_{\lambda}^{0}(\R)) \times D(V_{0}^{\mu}(\R))}{S(V_{\lambda}^{0}(\R)) \times D(V_{0}^{\mu}(\R))}
}
\end{split}
\end{align}
commutative up to homotopy, where $p_{2} \circ f^{S^{1}}: \tilde{D}(W_{0}(\R))/\tilde{S}(W_{0}(\R))\to N^{S^{1}}/L^{S^{1}}$ and $p_{2} \circ f^{S^{1}}: \tilde{D}(W_{0}(\R))/\tilde{S}(W_{0}(\R))\to \frac{D(V_{\lambda}^{0}(\R)) \times D(V_{0}^{\mu}(\R))}{S(V_{\lambda}^{0}(\R)) \times D(V_{0}^{\mu}(\R))}
$ are maps naturally induced by the same map $
p_{2} \circ f^{S^{1}} : W_0 (\R) \to V^\mu_\lambda(\R) $.
Note also an obvious commutative diagram
\begin{align}
\begin{split}
\label{eq: htpy diagram2 Tira}
\xymatrix{
\tilde{D}(W_{0}(\R))/\tilde{S}(W_{0}(\R)) \ar[r]^-{p_{2} \circ f^{S^{1}}} \ar[rd]_{p_{V_{\lambda}^{0}(\R)} \circ f^{S^{1}}\quad}
&\frac{D(V_{\lambda}^{0}(\R)) \times D(V_{0}^{\mu}(\R))}{S(V_{\lambda}^{0}(\R)) \times D(V_{0}^{\mu}(\R))} \ar[d]^-{\simeq} \\
&D(V_{\lambda}^{0}(\R))/S(V_{\lambda}^{0}(\R)).
}
\end{split}
\end{align}
Defining $\varphi$ as the composition of the vertical arrows in \eqref{eq: htpy diagram1 Tira} and \eqref{eq: htpy diagram2 Tira}, we obtain a homotopy commutative diagram \eqref{eq: htpy diagram Tira}.
\end{proof}

\subsection{Proof of \cref{theo: main theo}}
\label{subsection: Proof of first main theo}

Now we may start proving \cref{theo: main theo}.
Recall that all (co)homology are taken with $\F=\Z/2$-coefficients throughout this paper.

\begin{proof}[Proof of \cref{theo: main theo}]
Let us consider the following commutative diagram obtained by restricting the families relative Bauer--Furuta invariant onto the $S^{1}$-fixed-point sets:
\begin{align}
\begin{split}
\label{eq: diagram 0}
\xymatrix{
W_{0}^{+_{B}} \ar[r]^-{f}&W_{1}^{+_{B}} \wedge_{B} \underline{I_{\lambda}^{\mu}} \\
W_{0}(\R)^{+_{B}} \ar[u]^{i_{0}} \ar[r]^-{f^{S^{1}}} 
&W_{1}(\R)^{+_{B}} \wedge_{B} \underline{(I_{\lambda}^{\mu})^{S^{1}}} \ar[u]^{i_{1}}.
}
\end{split}
\end{align}
Here $i_{0}, i_{1}$ denote the inclusion maps.

The following lemma can be checked in a straightforward manner, and we omit the proof.

\begin{lem}
\label{leq: some quotient well def}
Let $W, W' \to B$ be vector bundles over $B$ and $I$ a pointed space. 
Denote by $\Th(W)$ the Thom space of $W$.
Then we have:
\begin{enumerate}
\item The identity map $W \times I \to W \times I$ induces a well-defined map
\[
W^{+_{B}} \wedge_{B} \underline{I} \to \Th(W) \wedge I.
\]
\item Assume that we have a fiberwise pointed map $\varphi : (W')^{+_{B}} \to W^{+_{B}} \wedge_{B} \underline{I}$.
Then $\varphi$ induces a well-defined map
\[
\Th(W') \to \Th(W) \wedge I.
\]
\item For a natural number $n$, the identity map $W \oplus \underline{\R^{n}} \to W \oplus \underline{\R^{n}}$ induces a well-defined homeomorphism
\[
\Th(W \oplus \underline{\R^{n}}) \to \Th(W) \wedge S^{n}.
\]
\end{enumerate}
\end{lem}
From this, it follows that the commutative diagram \eqref{eq: diagram 0} induces the following commutative diagram:
\begin{align}
\label{eq: diagram 2}
\xymatrix{
\Th(W_{0}) \ar[r]^-{f}&\Th(W_{1}) \wedge I_{\lambda}^{\mu} \\
\Th(W_{0}(\R)) \ar[u]^{i_{0}}  \ar[r]^-{f^{S^{1}}} 
&\Th(W_{1}(\R)) \wedge (I_{\lambda}^{\mu})^{S^{1}}\ar[u]^{i_{1}}.
}
\end{align}
Applying the functor $\tilde{H}_{S^{1}}^{\ast}( \cdot ;\F)$, we obtain the commutative diagram
\begin{align}
\begin{split}
\label{eq: diagram 3}
\xymatrix{
\tilde{H}_{S^{1}}^{\ast}(\Th(W_{0})) \ar[d]_{i_{0}^{\ast}}
&\tilde{H}_{S^{1}}^{\ast}(\Th(W_{1}) \wedge I_{\lambda}^{\mu})\ar[l]_-{f^{\ast}} \ar[d]_{i_{1}^{\ast}} \\
\tilde{H}_{S^{1}}^{\ast}(\Th(W_{0}(\R)))
&\tilde{H}_{S^{1}}^{\ast}(\Th(W_{1}(\R)) \wedge (I_{\lambda}^{\mu})^{S^{1}}) \ar[l]_-{(f^{S^{1}})^{\ast}} .
}
\end{split}
\end{align}
We shall derive the divisibility of the Euler classes of some bundles using the diagram \eqref{eq: diagram 3}.
To do this in our situation, we will take a cohomology class 
\[
\eta \in \tilde{H}_{S^{1}}^{\ast}(\Th(W_{1}) \wedge I_{\lambda}^{\mu})
\]
as follows.
Henceforth, as an abbreviation, we write $d$ for $d(Y, \lambda, \mu, g, \frakt) \in \Z$.
Set 
\[
s = \dim V_{\lambda}^{0}(\R).
\]
By \cref{lem: expression of omega}, there exists a cohomology class 
\[
\omega \in \tilde{H}_{S^{1}}^{d}(I_{\lambda}^{\mu})
\]
satisfying the equality \eqref{eq: expression of iomega}.
Setting
\begin{align}
\label{eq: d prime}
d' = (d-s)/2,
\end{align}
we have
\begin{align}
\label{eq: expression of iomega2}
i^{\ast}\omega = [V^{0}_{\lambda}(\R)^{+}] \otimes U^{d'},
\end{align}
where i is the inclusion from $(I_{\lambda}^{\mu})^{S^1}$ to $I_{\lambda}^{\mu}$.

Here recall an elementary observation used in the K\"unneth formula for the reduced cohomology.
Let $X_{1}, X_{2}$ be based $S^{1}$-spaces and $p_{1} : (X_{1} \times X_{2}, \ast \times X_{2}) \to (X_{1}, \ast)$ and $p_{2} : (X_{1} \times X_{2}, X_{1} \times \ast) \to (X_{2}, \ast)$ be the projections.
For cohomology classes $\gamma_{i} \in \tilde{H}_{S^{1}}^{\ast}(X_{i}) \cong H_{S^{1}}^{\ast}(X_{i},\ast)$, the cohomology class $p_{1}^{\ast}\gamma_{1} \cup p_{2}^{\ast}\gamma_{2}$ can be thought of an element of
\[
H_{S^{1}}^{\ast}(X_{1} \times X_{2}, (X_{1} \times \ast) \cup (\ast \times X_{2}))
\cong \tilde{H}_{S^{1}}^{\ast}(X_{1} \wedge X_{2}).
\]

Now we go back to the diagram \eqref{eq: diagram 3} and apply the above observation to $\Th(W_{1}) \wedge I_{\lambda}^{\mu}$.
Let 
\[
p_{1} : (\Th(W_{1}) \times I_{\lambda}^{\mu}, \ast \times I_{\lambda}^{\mu}) \to (\Th(W_{1}), \ast)
\]
and 
\[
p_{2} : (\Th(W_{1}) \times I_{\lambda}^{\mu}, \Th(W_{1}) \times \ast) \to (I_{\lambda}^{\mu}, \ast)
\]
be the projections.
Then we obtain a cohomology class
\begin{align}
\label{eq: eta def}
\eta := p_{1}^{\ast}\tau_{S^{1}}(W_{1}) \cup p_{2}^{\ast}\omega \in \tilde{H}_{S^{1}}^{\ast}(\Th(W_{1}) \wedge I_{\lambda}^{\mu}).
\end{align}
We obtain
\begin{align}
\label{eq: main rel}
i_{0}^{\ast} f^{\ast} \eta
= (f^{S^{1}})^{\ast} i_{1}^{\ast} \eta
\end{align}
from the commutativity of the diagram \eqref{eq: diagram 3}.
Let us write down two sides of this relation \eqref{eq: main rel} in detail and extract a constraint on $H^{+}(E)$.

First, by \cref{Thom}, the equivariant Thom isomorphism with coefficients $\F$, there exists a cohomology class $\theta \in H_{S^{1}}^{\ast}(B)$ such that
\begin{align}
\label{eq: first pt for degree}
\tau_{S^{1}}(W_{0}) \cup \pi_{W_{0}}^{\ast}\theta = f^{\ast} \eta,
\end{align}
where $\pi_{W_{0}} : W_{0} \to B$ denotes the projection.
This cohomology class $\theta$ is an analog of the cohomological mapping degree of $f$ used to extract ordinary-cohomological information from the families Bauer--Furuta invariant of a family of closed $4$-manifolds.

Next, let us note the following elementary observation on Thom classes.
Let $W \oplus W' \to B$ be vector bundles decomposed into a direct sum.
Let $S^{1}$ act on a given vector bundle as the trivial action or the multiplication of complex numbers according to whether the bundle is a real or complex vector bundle.
Let $i : W \inc W\oplus W'$ be the inclusion.
A basic formula used below is
\begin{align}
\label{eq: general Thom res Euler}
i^{\ast}\tau_{S^{1}}(W \oplus W')
= \tau_{S^{1}}(W) \cup \pi_{W}^{\ast}e_{S^{1}}(W'),
\end{align}
which holds in $\tilde{H}_{S^{1}}^{\ast}(\Th(W))$.

By the previous paragraph, more precisely the formula \eqref{eq: general Thom res Euler}, we have 
\begin{align}
\label{eq: i zeroast}
i_{0}^{\ast}\tau_{S^{1}}(W_{0}) = \tau_{S^{1}}(W_{0}(\R)) \cup \pi_{W_{0}(\R)}^{\ast}e_{S^{1}}(W_{0}(\C)).
\end{align}
It follows from \eqref{eq: first pt for degree} and \eqref{eq: i zeroast} that
\begin{align}
\label{eq: izeroast fast}
i_{0}^{\ast} f^{\ast} \eta = \tau_{S^{1}}(W_{0}(\R)) \cup \pi_{W_{0}(\R)}^{\ast}(e_{S^{1}}(W_{0}(\C)) \cup \theta).
\end{align}

Next, we calculate the right-hand side of \eqref{eq: main rel}.
By abuse of notation we denote also by $p_{1}, p_{2}$ the projections
\begin{align*}
&p_{1} : (\Th(W_{1}(\R)) \times (I_{\lambda}^{\mu})^{S^{1}}, \ast \times (I_{\lambda}^{\mu})^{S^{1}}) \to (\Th(W_{1}(\R)), \ast),\\
&p_{2} : (\Th(W_{1}(\R)) \times (I_{\lambda}^{\mu})^{S^{1}}, \Th(W_{1}(\R)) \times \ast) \to ((I_{\lambda}^{\mu})^{S^{1}}, \ast)
\end{align*}
respectively.
Let 
\begin{align*}
&\iota_{1} : \Th(W_{1}(\R)) \inc \Th(W_{1})\\
\end{align*}
be the inclusion.
Then, by \eqref{eq: expression of iomega2} and \eqref{eq: general Thom res Euler}, we have that
\begin{align}
\begin{split}
\label{eq: ioneast}
i_{1}^{\ast} \eta
=& p_{1}^{\ast}\iota_{1}^{\ast}\tau_{S^{1}}(W_{1}) \cup p_{2}^{\ast} i^{\ast}\omega \\
=& p_{1}^{\ast}\tau_{S^{1}}(W_{1}(\R)) \cup p_{1}^{\ast}\pi_{W_{1}(\R)}^{\ast}e_{S^{1}}(W_{1}(\C)) \cup p_{2}^{\ast}([V_{\lambda}^{0}(\R)^{+}] \otimes U^{d'})
\end{split}
\end{align}
in $\tilde{H}_{S^{1}}^{\ast}(\Th(W_{1}(\R)) \wedge (I_{\lambda}^{\mu})^{S^{1}})$.
Let 
\begin{align*}
\Phi : \tilde{H}_{S^{1}}^{\ast}(\Th(W_{1}(\R)) \wedge V_{\lambda}^{0}(\R)^{+}) \to \tilde{H}_{S^{1}}^{\ast}(\Th(W_{1}(\R)) \wedge (I_{\lambda}^{\mu})^{S^{1}}).
\end{align*}
be the isomorphism induced from the homotopy equivalence $\varphi : N^{S^{1}}/L^{S^{1}} \to V_{\lambda}^{0}(\R)^{+}$ obtained in \cref{lem: htpy NL},
where we identify $(I_{\lambda}^{\mu})^{S^{1}}$ with $N^{S^{1}}/L^{S^{1}}$ using an obvious homeomorphism.
\Cref{lem: htpy NL} implies that we have the commutative diagram
\begin{align}
\begin{split}
\label{eq: coh diagram1}
\xymatrix{
\tilde{H}_{S^{1}}^{\ast}(\Th(W_{0}(\R)))
&\tilde{H}_{S^{1}}^{\ast}(\Th(W_{1}(\R)) \wedge (I_{\lambda}^{\mu})^{S^{1}}) \ar[l]_-{(f^{S^{1}})^{*}}\\
&\tilde{H}_{S^{1}}^{\ast}(\Th(W_{1}(\R)) \wedge V_{\lambda}^{0}(\R)^{+}) \ar[ul]^{((\id_{W_{1}(\R)} \oplus p_{V_{\lambda}^{0}(\R)}) \circ f^{S^{1}})^{*}\quad\qquad} \ar[u]^-{\Phi}.
}
\end{split}
\end{align}
Note that we have an isomorphism
\[
\Psi : \tilde{H}_{S^{1}}^{\ast}(\Th(W_{1}(\R))\wedge V_{\lambda}^{0}(\R)^{+})
\to \tilde{H}_{S^{1}}^{\ast}(\Th(W_{1}(\R) \oplus \underline{V_{\lambda}^{0}(\R)})).
\]
induced from a natural homeomorphism in  \cref{leq: some quotient well def}~(3).
Via $\Psi$, we identify the domain and codomain of $\Psi$.
It follows from \eqref{eq: ioneast} that
\begin{align}
\label{eq: expression tempo}
\Phi^{-1} \circ i_{1}^{\ast}\eta=
\tau_{S^{1}}(W_{1}(\R) \oplus \underline{V_{\lambda}^{0}(\R)}) \cup \pi_{W_{1}(\R) \oplus \underline{V_{\lambda}^{0}(\R)}}^{\ast}(e_{S^{1}}(W_{1}(\C)) \cdot U^{d'}),
\end{align}
where $\cdot U^{d'}$ denotes the action of $U^{d'} \in \tilde{H}_{S^{1}}^{\ast}(S^{0}) \cong \F[U]$ on ${H}_{S^{1}}^{\ast}(B)$.
Recall that $f^{S^{1}}$ is obtained as the restriction of a fiberwise linear map \eqref{eq: fix f 2}.
Moreover, the map \eqref{eq: fix f 2}, which induces the map 
\[
((\id_{W_{1}(\R)} \oplus p_{V_{\lambda}^{0}(\R)}) \circ f^{S^{1}})^{*} : \tilde{H}_{S^{1}}^{\ast}(\Th(W_{1}(\R)) \wedge V_{\lambda}^{0}(\R)^{+})
\to \tilde{H}_{S^{1}}^{\ast}(\Th(W_{0}(\R)))
\]
in the diagram \eqref{eq: coh diagram1}, is a fiberwise linear injection and its fiberwise cokernel is isomorphic to $H^{+}(E)$.
It follows from this combined with \eqref{eq: general Thom res Euler}, \eqref{eq: coh diagram1}, and \eqref{eq: expression tempo} that
\begin{align}
\begin{split}
\label{eq: f fix ast i one ast}
(f^{S^{1}})^{\ast} i_{1}^{\ast} \eta
&= ((\id_{W_{1}(\R)} \oplus p_{V_{\lambda}^{0}(\R)}) \circ f^{S^{1}})^{\ast}\circ \Phi^{-1} \circ i_{1}^{\ast} \eta\\
&= \tau_{S^{1}}(W_{0}(\R)) \cup \pi_{W_{0}(\R)}^{\ast} (e_{S^{1}}(H^{+}(E)) \cup e_{S^{1}}(W_{1}(\C)) \cdot U^{d'}).
\end{split}
\end{align}

Since the Thom class $\tau_{S^{1}}(W_{0}(\R)) \in \tilde{H}_{S^{1}}^{\ast}(\Th(W_{0}(\R)))$ is a generator of $\tilde{H}_{S^{1}}^{\ast}(\Th(W_{0}(\R)))$ as an $H_{S^{1}}^{*}(B)$-module, it follows from \eqref{eq: main rel}, \eqref{eq: izeroast fast}, and \eqref{eq: f fix ast i one ast} that
\begin{align}
\label{eq: main divisibility}
e_{S^{1}}(W_{0}(\C)) \cup \theta
= e_{S^{1}}(H^{+}(E)) \cup e_{S^{1}}(W_{1}(\C)) \cdot U^{d'}.
\end{align}
This is an equality in $H_{S^{1}}^{\ast}(B)$,
and is the desired divisibility of Euler classes.

Set $m := \rank_{\C}W_{0}(\C)$ and $n:= \rank_{\C}W_{1}(\C)$.
Recall that the $S^{1}$-action on $W_{i}(\C)$ is given by the scalar multiplication.
Then the equivariant Euler class is written in terms of (non-equivariant) Chern classes, which is actually one of ways to define the Chern classes:
\begin{align*}
&e_{S^{1}}(W_{0}(\C)) = \sum_{i=0}^{m} c_{m-i}(W_{0}(\C)) \otimes U^{i},\\
&e_{S^{1}}(W_{1}(\C)) = \sum_{j=0}^{n} c_{n-j}(W_{1}(\C)) \otimes U^{j}
\end{align*}
in $H_{S^{1}}^{\ast}(B;\Z) \cong H^{\ast}(B;\Z) \otimes H_{S^{1}}^{\ast}(\pt;\Z)$.
Taking mod $2$, we obtain similar equalities involving Stiefel--Whitney classes in $H_{S^{1}}^{\ast}(B) = H_{S^{1}}^{\ast}(B;\F)$.
It follows from this combined with \eqref{eq: main divisibility} that 
\begin{align}
\label{eq: div substi}
\left(\sum_{i=0}^{m} w_{2(m-i)}(W_{0}(\C)) \otimes U^{i} \right) \cup \theta
= e_{S^{1}}(H^{+}(E)) \cup \left(\sum_{j=0}^{n} w_{2(n-j)}(W_{1}(\C)) \otimes U^{j} \right) \cup U^{d'}.
\end{align}
Here note that we have $e_{S^{1}}(H^{+}(E)) = w_{b^{+}}(H^{+}(E)) \otimes 1 \in H^{*}(B) \otimes H_{S^{1}}^{\ast}(\pt)$ since the action of $S^{1}$ on $H^{+}(E)$ is trivial.
Now let us use the assumption that $w_{b^{+}}(H^{+}(E)) \neq 0$.
Setting $k := \deg_{U}\theta \geq 0$ and comparing the $U$-degree highest terms in the equality \eqref{eq: div substi}, we obtain
\[
\theta_{0} \cdot U^{m+k} = w_{b^{+}}(H^{+}(E)) \cdot U^{n+d'},
\]
where $\theta_{0} \in H^{*}(B)$ is a non-zero cohomology class. 
Thus we have that
$m+k =n+d'$, and hence $m \leq n+d'$.
By \eqref{eq: dirac index formula}, this inequality is equivalent to
\[
\frac{c_{1}(\fraks)^{2} - \sigma(X)}{8} + n(Y,\frakt,g)+ \dim_{\C} V^{0}_{\lambda}(\C) \leq d'.
\]
From the definition of the Fr{\o}yshov invariant and the definition of $d'$, which are \eqref{def Fro h d} and \eqref{eq: d prime} respectively, this is equivalent to the desired inequality \eqref{eq: main ineq in main thm}.
This completes the proof of \cref{theo: main theo}.
\end{proof}

\begin{rem}
\label{rem: why not local system}
Baraglia~\cite{Ba19} used local coefficient systems with fiber $\Z$ to derive his constraint \cite[Theorem~1.1]{Ba19}.
As a result, he obtained a constraint described in terms of the Euler class of $H^{+}(E)$ living in a certain cohomology with local coefficient, not $w_{b^{+}}(H^{+}(E))$.
\Cref{theo: main theo} is an analog of the mod $2$ version of his constraint.
Here we explain the reason why we cannot use such local coefficients and use $\F$-coefficients instead in this paper.
Given an $S^{1}$-vector bundle $W \to B$, 
to use the (equivariant) Thom isomorphism for $W$ with a certain local coefficient induced from a local system on the base space, we need to consider the relative cohomology $H^{\ast}_{S^{1}}(D(W), S(W))$, rather than $\tilde{H}_{S^{1}}^{\ast}(\Th(W))$.
This is just because there is no obvious way to define a local system on $\Th(W)$ induced from a local system on the base space $B$.
To use relative cohomologies, we need to have a map between pairs
\begin{align*}
f : (D(W_{0}), S(W_{0}))
\to (W_{1}, W_{1}\setminus\{0\}) \times (N, L)
\end{align*}
instead of \eqref{eq: starting point}.
But we could not figure out whether we can obtain such a map as the families relative Bauer--Furuta invariant, because it seems essential to cut the domain of $f$ by the compact sets $\tilde{K}_{1}, \tilde{K}_{2}$ in \cref{lem: ensure the existence of Conley index} to obtain an appropriate index pair $(N,L)$.
\end{rem}

\subsection{Proof of \cref{theo: main theo2}}
\label{subsection: Proof of first main theo2}

The proof of \cref{theo: main theo2} is quite similar to the proof of \cref{theo: main theo}.
Here let us summarize major difference of the settings:
\begin{itemize}
\item The families $Pin(2)$-equivariant relative Bauer--Furuta invariant for the smooth family $(X,\fraks) \to E \to B$ is given as a fiberwise $Pin(2)$-equivariant map between families of pointed $Pin(2)$-spaces parametrized over $B$:
\begin{align}
\label{eq: starting point}
f : W_{0}^{+_{B}} 
\to W_{1}^{+_{B}} \wedge_{B} \underline{I_{\lambda}^{\mu}}.
\end{align}
\item The vector bundles $W_{0}, W_{1}$ are the direct sums of real vector bundles $W_{i}(\tilR)$ and quaternionic vector bundle $W_{i}(\quat)$ over $B$.
Here $\Pin(2)$ acts on $W_{i}(\tilR)$ as the $\pm1$-multiplication and on $W_{i}(\quat)$ as the scalar multiplication of quaternions. The pointed space $\underline{I_{\lambda}^{\mu}}$ is a $Pin(2)$-equivariant Conley index.
\item We have
\begin{align}
\label{eq: dirac index formula 2}
\rank_{\quat}W_{0}(\quat) - \rank_{\quat}W_{1}(\quat) 
= \frac{- \sigma(X)}{16} + \frac{n(Y,\frakt,g)}{2}+ \dim_{\quat} V^{0}_{\lambda}(\quat).
\end{align}
\item $\Pin(2)$-equivariant cohomology and $\Pin(2)$-equivariant Thom and Euler classes are used, instead of $S^{1}$-equivariant cohomology, Thom classes, and Euler classes.
\end{itemize}

\begin{proof}[Proof of \cref{theo: main theo2}]
Considering the restriction of a finite-dimensional approximation $f$ to the $S^{1}$-fixed part,
we have a commutative diagram
\begin{align}
\begin{split}
\label{eq: diagram 3}
\xymatrix{
\tilde{H}_{\Pin(2)}^{\ast}(\Th(W_{0})) \ar[d]_{i_{0}^{\ast}}
&\tilde{H}_{\Pin(2)}^{\ast}(\Th(W_{1}) \wedge I_{\lambda}^{\mu})\ar[l]_-{f^{\ast}} \ar[d]_{i_{1}^{\ast}} \\
\tilde{H}_{\Pin(2)}^{\ast}(\Th(W_{0}(\tilR)))
&\tilde{H}_{\Pin(2)}^{\ast}(\Th(W_{1}(\tilR)) \wedge (I_{\lambda}^{\mu})^{S^{1}}) \ar[l]_-{(f^{S^{1}})^{\ast}} .
}
\end{split}
\end{align}
Let $\omega$ be one of
\begin{align*}
&\omega_{a} \in \tilde{H}_{S^{1}}^{a}(I_{\lambda}^{\mu}),\\
&\omega_{b} \in \tilde{H}_{S^{1}}^{b+1}(I_{\lambda}^{\mu}),\\
&\omega_{c} \in \tilde{H}_{S^{1}}^{c+2}(I_{\lambda}^{\mu})
\end{align*}
in \cref{lem: expression of omega spin},
and define $\eta \in \tilde{H}_{S^{1}}^{\ast}(\Th(W_{1}) \wedge I_{\lambda}^{\mu})$ using this $\omega$ as well as \eqref{eq: eta def}.
Repeating the proof of \cref{theo: main theo} using the diagram \eqref{eq: diagram 3} and this $\eta$, we obtain
\begin{align}
&e_{\Pin(2)}(W_{0}(\quat)) \cup \theta
= e_{\Pin(2)}(H^{+}(E)) \cup e_{\Pin(2)}(W_{1}(\quat)) \cdot v^{(a-s)/4},\label{eq: main divisibility spin1}\\
&e_{\Pin(2)}(W_{0}(\quat)) \cup \theta
= e_{\Pin(2)}(H^{+}(E)) \cup e_{\Pin(2)}(W_{1}(\quat)) \cdot qv^{(b-s)/4},\label{eq: main divisibility spin2}\\
&e_{\Pin(2)}(W_{0}(\quat)) \cup \theta
= e_{\Pin(2)}(H^{+}(E)) \cup e_{\Pin(2)}(W_{1}(\quat)) \cdot q^{2}v^{(c-s)/4},\label{eq: main divisibility spin3}
\end{align}
according to the choice of $\omega$, as well as \eqref{eq: main divisibility}.
Here $\theta$ is an element of $H_{\Pin(2)}^{\ast}(B)$.

By an argument by Baraglia in the proof of \cite[Theorem~5.1]{Ba19}, we have that
\begin{align}
&e_{\Pin(2)}(W_{0}(\quat)) = \sum_{i=0}^{m} c_{2m-2i}(W_{0}(\quat)) \otimes v^{i}, \label{eq quat euler1}\\
&e_{\Pin(2)}(W_{1}(\quat)) = \sum_{j=0}^{n} c_{2n-2j}(W_{1}(\quat)) \otimes v^{j} \label{eq quat euler2}
\end{align}
in $H_{\Pin(2)}^{\ast}(B) \cong H^{\ast}(B) \otimes H_{\Pin(2)}^{\ast}(\pt)$, where $m = \rank_{\quat} W_{0}(\quat)$ and $n = \rank_{\quat} W_{1}(\quat)$.
Moreover, an argument by Baraglia in the proof of \cite[Corollary~5.2]{Ba19}, we have that
\begin{align}
\begin{split}
\label{eq quat euler H plus}
e_{\Pin(2)}(H^{+}(E)) 
&= w_{b^{+},\Pin(2)}(H^{+}(E))\\
&= w_{b^{+}}(H^{+}(E)) \otimes 1 + w_{b^{+}-1}(H^{+}(E)) \otimes q + w_{b^{+}-2}(H^{+}(E)) \otimes q^{2}
\end{split}
\end{align}
in $H_{\Pin(2)}^{\ast}(B)$.

Now we argue according to the non-vanishing of $w_{\bullet}(H^{+}(E))$ for $\bullet=b^{+}, b^{+}-1, b^{+}-2$.
First let us assume that $w_{b^{+}}(H^{+}(E)) \neq 0$.
In this case, we take $\omega_{c}$ as $\omega$.
Let us substitute \eqref{eq quat euler1}, \eqref{eq quat euler2} and \eqref{eq quat euler H plus} for various Euler classes in \eqref{eq: main divisibility spin3}.
Then one may see that the right-hand side of \eqref{eq: main divisibility spin3} contains the term $w_{b^{+}}(H^{+}(E)) \otimes q^{2}v^{(c-s)/4}$, which is the $v$-degree  highest non-zero term of the form $x \otimes q^{2}v^{k}$, where $x \in H^{\ast}(B)$ and $k \geq 0$.
Therefore the left-hand side of \eqref{eq: main divisibility spin3} should also contain a non-zero term of the form $x \otimes q^{2}v^{k}$.
This is equivalent to the existence of a non-zero term of the form $x \otimes q^{2}v^{k}$ in $\theta$.
Let $k_{c} \geq 0$ be the maximum of such $k$.
Then it follows that
\[
\theta_{0} \otimes q^{2}v^{m+k_{c}} = w_{b^{+}}(H^{+}(E)) \otimes q^{2}v^{n+(c-s)/4},
\]
where $0 \neq \theta_{0} \in H^{*}(B)$.
Thus we have $m-n \leq (c-s)/4$.
This combined with the definition of $\gamma$, given in \eqref{eq: def gamma}, implies the inequality \eqref{eq: main ineq in main thm2}.

Next let us assume that $b^{+}(X)>0$ and $w_{b^{+}-1}(H^{+}(E)) \neq 0$.
In this case, we take $\omega_{b}$ as $\omega$.
After substituting \eqref{eq quat euler1}, \eqref{eq quat euler2} and \eqref{eq quat euler H plus} for the Euler classes in \eqref{eq: main divisibility spin2},
the right-hand side of \eqref{eq: main divisibility spin2} contains the term $w_{b^{+}-1}(H^{+}(E)) \otimes q^{2}v^{(b-s)/4}$, which is the $v$-degree  highest non-zero term of the form $x \otimes q^{2}v^{k}$.
Arguing exactly as in the above paragraph, we obtain $m-n \leq (b-s)/4$, which implies the inequality \eqref{eq: main ineq in main thm3}.

Similarly, the inequality \eqref{eq: main ineq in main thm4} is deduced from the assumption that $b^{+}(X)>1$ and $w_{b^{+}-2}(H^{+}(E)) \neq 0$ by taking $\omega_{a}$ as $\omega$. 
This completes the proof of \cref{theo: main theo2}.
\end{proof}

\begin{rem}
A reader may wonder whether $\Pin(2)$-equivariant $K$-theory can be used to extract a constraint
of smooth families of spin $4$-manifolds with boundary.
We predict that it should be able to be established as a general statement using Manolescu's invariant $\kappa$ introduced in \cite{Ma14} instead of $\alpha,\beta,\gamma$.
The reason why we do not include such a study in this paper is that we could not find a potential application like \cref{cor: rel Diff Homeo,cor: Diff Homeo absolute} detected using a $K$-theoretic constraint.
\Cref{cor: rel Diff Homeo,cor: Diff Homeo absolute} follows from the existence of non-smoothable families (\cref{theo: main app}), but non-smoothability of families of that kind cannot be detected using a $K$-theoretic constraint.
For the examples of non-smoothable families $E$ given in \cref{subsection: proof of main app}, the associated bundles $H^{+}(E)$ do not admit $K$-theory orientation, and the $K$-theoretic Euler class cannot make sense for them.
(One way to get $K$-orientability is tensoring with $\C$, but $H^{+}(E) \otimes \C$ are trivial in those examples, and we cannot extract any information.)
\end{rem}

\section{Applications}
\label{section: Applications}

In this \lcnamecref{section: Applications} we consider applications of \cref{theo: main theo,theo: main theo2} mainly to the existence of non-smoothable families of 4-manifolds with boundary, stated as \cref{theo: main app}.
We also describe consequences of the the existence of non-smoothable families about comparison between various diffeomorphism groups and homeomorphism groups of 4-manifolds with boundary in this \lcnamecref{section: Applications}.

\subsection{Topological spin and spin$^{c}$ structure}
\label{subsection: Topological spinc structure}

To apply our main theorems, \cref{theo: main theo,theo: main theo2}, to concrete families of $4$-manifolds, we need to lift structure group from the homeomorphism group to the automorphism group of a topological spin or spin$^{c}$ structure.
This problem has already appeared also in the study of families of closed $4$-manifolds \cite{BK19,KKN19,Ba19,KN20}.
Although there is no major difference between closed manifolds and manifolds with boundary on this problem,
for readers' convenience, 
we recall the notion of a topological spin or spin$^{c}$ structure and a sufficient condition for the above lifting problem.
We mainly refer to \cite[Subsection~4.2]{BK19} for the detail.

First recall the definition of a microbundle.
Let $B$ be a topological space.
A {\it microbundle} $\xi$ over $B$ with fiber dimension $n \geq 0$ consists of data $(E,B,i,p)$, symbolically denoted by
\[
\xi = \{ B \xrightarrow{i} E \xrightarrow{p} B \},
\]
where $E$ is a topological space, $i,p$ are continuous maps satisfying that $p \circ i = \id_{B}$, and $E$ is supposed to be locally trivial around the image of $i$.
Namely, for each $b \in B$, there exist an open neighborhood $O$ of $b$ in $B$, an open neighborhood $U$ of $i(b)$ of $E$, and an homeomorphism $\phi : U \to O \times \R^{n}$ such that $p|_{U} = p_{2} \circ \phi$, where $p_{2} : O \times \R^{n} \to \R^{n}$ is the projection.

The following Kister's theorem~\cite{Ki64} (and its extension to a paracompact base space by Holm~\cite{Ho67}) are fundamental in the study of microbundles.
This states that one can find a fiber bundle which represents a given microbundle, and such representatives are unique up to isomorphism:

\begin{theo}[\cite{Ki64,Ho67}]
\label{theo: Holm}
If $B$ is paracompact, any microbundle $\xi=\{ B \xrightarrow{i} E \xrightarrow{p} B \}$ over $B$ is represented by a fiber bundle which is unique up to isomorphism.
Namely, if we denote by $n$ the fiber dimension of $\xi$, there exists an open neighborhood $U \subset E$ of the image of $i$ such that $p|_{U} : U \to B$ is a fiber bundle over $B$ with fiber $\R^{n}$, and such a fiber bundle is unique up to isomorphism.
\end{theo}

For $n \geq 1$, denote by $Top(n)$ the group of homeomorphisms of $\R^{n}$ preserving the origin and denote by $STop(n)$ the subgroup of $Top(n)$ consisting of homeomorphisms preserving the orientation of $\R^{n}$.
Given a topological $n$-manifold $X$ without boundary, the {\it tangent microbundle}
\[
\tau X = \{ X \xrightarrow{\Delta} X \times X \xrightarrow{p_{1}} X\}
\]
is associated, where $\Delta$ is the diagonal map and $p_{2}$ is the projection to the first factor.
By \cref{theo: Holm}, a principal $Top(n)$-bundle over $X$ is associated to $\tau X$, and this fiber bundle is unique up to isomorphism.
Henceforth we use the notation $\tau X$ also to indicate the associated fiber bundle and say that $\tau X$ has structure group $Top(n)$ if there is no risk of confusion.
If $X$ is oriented, we obtain the oriented tangent microbundle, where the associated fiber bundle has structure group $STop(n)$.

Similarly, if we have a topological fiber bundle $X \to E \to B$ which has structure group $\Homeo(X)$, the {\it vertical tangent microbundle} $\tau(E/B)$ is associated:
\[
\tau(E/B) = \{ E \xrightarrow{\Delta} E \times_{B} E \xrightarrow{p_{1}} E\}.
\]
A principal $Top(n)$-bundle over $E$ is associated to $\tau(E/B)$ by \cref{theo: Holm}.
If the structure group of $X \to E \to B$ reduces to $\Homeo^{+}(X)$, we have the oriented vertical tangent microbundle $\tau(E/B)$.

Let us assume $n \geq 2$.
Since $STop(n)$ is connected and the natural map $GL^{+}_{n}(\R) \inc STop(n)$ is known to induce an isomorphism of fundamental groups, we obtain a unique connected double covering of $STop(n)$, which we denote by $SpinTop(n)$.
If the oriented tangent microbundle $\tau X$ admits a lift to a principal $SpinTop(n)$-bundle along the covering $SpinTop(n) \to STop(n)$,
we call such a list a {\it (topological) spin structure} of $X$.
For a smooth manifold, usual (i.e. smooth) spin structures are naturally in a bijective correspondence with topological spin structures.
A fiberwise topological spin structure of a fibrewise oriented topological fiber bundle $X \to E \to B$ is also defined in a similar vein, using the vertical tangent microbundle instead.

Next, let us define a topological group $Spin^{c}Top(n)$ by
\[
Spin^{c}Top(n)=(SpinTop(n) \times S^{1})/(\Z/2),
\]
where the action of $\Z/2$ is the diagonal action, and the action of $\Z/2$ on $SpinTop(n)$ is the covering transformation of the double cover $SpinTop(n) \to STop(n)$.
We have a natural map $Spin^{c}Top(n) \to STop(n)$ by ignoring the second factor $S^{1}$.
If $\tau X$ admits a lift to a principal $Spin^{c}Top(n)$-bundle along this map $Spin^{c}Top(n) \to STop(n)$, we call such a list a {\it (topological) spin$^{c}$ structure} of $X$.
A fiberwise topological spin$^{c}$ structure of a fibrewise oriented topological fiber bundle $X \to E \to B$ is also similarly defined.

If we consider an oriented topological $n$-manifold $X$ with boundary instead, there is no major change about topological spin or spin$^{c}$ structure of $X$ described above.
Just note that we have the restriction of a given topological spin or spin$^{c}$ structure of $X$ onto the boundary $\del X$.
This is because a natural inclusion $STop(n-1) \inc STop(n)$ is covered by maps $SpinTop(n-1) \to SpinTop(n)$ and $Spin^{c}Top(n-1) \to Spin^{c}Top(n)$.
Similarly, we can consider fiberwise restriction of a topological spin or spin$^{c}$ structure of the vertical tangent microbundle of a fibrewise oriented topological fiber bundle $X \to E \to B$ which has structure group $\Homeo^{+}(X)$ or $\Homeo(X, \del)$.

Let us discuss a more specific situation.
Let $k \geq 1$, and 
let $(X_{0}, \fraks_{0})$ be a compact connected topological spin $4$-manifold with boundary and $(X_{1}, \fraks_{1}), \ldots, (X_{k}, \fraks_{k})$ be closed connected smooth spin $4$-manifolds.
Assume that $\del X_{0}$ is an integral homology $3$-sphere.
Let $f_{1}, \ldots, f_{k}$ be orientation-preserving diffeomorphisms of $X_{1}, \ldots, X_{k}$.
Assume that each of $f_{i}$ has a fixed $4$-disk $D_{i}$ in $X_{i}$ and preserves $\fraks_{i}$.
Form the connected sum $X := X_{0} \# X_{1} \# \cdots \# X_{k}$ by gluing $X_{i}$ around $D_{i}$, and
regard $f_{i}$ as a homeomorphism of $X$ extending by identity.
Let $X \to E \to T^{k}$ be the topological fiber bundle defined as the multiple mapping torus of the commuting homeomorphisms $f_{1}, \ldots, f_{k}$.

By arguments in pages 52--54 of \cite{BK19}, 
one can find representatives (in the sense of \cref{theo: Holm}) $U_{0}, U_{1}, \ldots, U_{k}$ of tangent microbundles of $X_{0}, X_{1}, \ldots, X_{k}$ with the following properties:
\begin{enumerate}
\item $U_{1}, \ldots, U_{k}$ are principal $GL^{+}(n,\R)$-bundles.
These are obtained as disk bundles of $TX_{1}, \ldots, TX_{k}$ with appropriate radii.
\item $U_{1}, \ldots, U_{k}$ can be glued together to form a representative $U'$, which is also a principal $GL^{+}(n,\R)$-bundle, of the tangent microbundle of $X_{1} \# \cdots \# X_{k}$;
\item $U_{0}$ can be glued with $U'$ to form a representative $U$, which is a principal $STop(n)$-bundle, of the tangent microbundle of $X$;
\item $f_{i} \times f_{i} : X \times X \to X \times X$ preserve $U$ for all $i = 1,\ldots, k$.
\end{enumerate}

Note that, on each of $U_{1}, \ldots, U_{k}$, we have a lift $f_{i}'$ of $f_{i}$ to $U_{i}$ as $GL^{+}(n,\R)$-bundle automorphism as follows.
First, $df_{i}(U_{i})$ gives rise to a fiber bundle over $X_{i}$ with fiber $\overset{\circ}{D^{n}}$, an open $n$-disk.
The fiber of $df_{i}(U_{i})$ and that of $U_{i}$ on the same point of $X_{i}$ are related by a linear transformation in $GL^{+}(n,\R)$, thus we have a $GL^{+}(n,\R)$-bundle isomorphism $\varphi_{i} : df_{i}(U_{i}) \to U_{i}$ covering the identity of $X_{i}$.
Then $f_{i}' := \varphi_{i} \circ df_{i} : U_{i} \to U_{i}$ is a lift of $f_{i}$.

Since $f_{i}'$ is the identity outside the support of $f_{i}$, we can extend $f_{i}'$ as an $STop(n)$-bundle automorphism over $X$ by gluing it with the identity.



\begin{lem}[{\cite[Proposition 4.18]{BK19}, \cite[Lemma~4.2]{KN20}}]
\label{lem: lift top spin}
Let $X \to E \to T^{k}$ be the mapping torus constructed as above.
Then $E$ admits a fiberwise spin structure such that the restriction of $E$ to a fiber is given by $\fraks = \fraks_{0} \# \fraks_{1} \# \cdots \# \fraks_{k}$, and that the restriction of $E$ to the fiberwise boundary is isomorphic to the trivial bundle of spin $3$-manifolds.

A similar statement holds also for topological spin$^{c}$ structure instead of topological spin structure.
\end{lem}

\begin{proof}
On the assertion about topological spin structure, one can adapt the proof of \cite[Proposition 4.18]{BK19}, but we give a proof for completeness.
By the construction of $f_{1}', \ldots, f_{k}'$ above, 
each $f_{i}'$ lifts to an $SpinTop(n)$-bundle automorphism $\tilde{f}_{i}$ as follows.
For $i = 1, \ldots, k$, let $P_{i} \to X_{i}$ denote a principal $\widetilde{GL}^{+}(n,\R)$-bundle corresponding to the smooth spin structure $\fraks_{i}$, and $P$ denote a principal $SpinTop(n)$-bundle covering $U$ corresponding to the spin structure $\fraks$.
Since $f_{i}$ is supposed to preserve $\fraks_{i}$, there exists a $\widetilde{GL}^{+}(n,\R)$-bundle automorphism $\hat{f}_{i} : P_{i} \to P_{i}$ covering the map $df_{i}$ between the ${GL}^{+}(n,\R)$-frame bundle.
Define a $\widetilde{GL}^{+}(n,\R)$-bundle automorphism $\tilde{f}_{i} : P_{i} \to P_{i}$ to be the pull-back of $\hat{f}_{i}$ under $\varphi_{i}$
It follows that $f_{i}'$ lifts to $\tilde{f}_{i}$, and that $\tilde{f}_{i}$ can be extended to an $SpinTop(n)$-automorphisms of $P$.

Now we see that $[\tilde{f}_{i}, \tilde{f}_{j}]=1$ for $i \neq j$.
Since $[f_{i}',f_{j}']=1$, we have that $[\tilde{f}_{i}, \tilde{f}_{j}]$ is a deck transformation of the double cover $P \to U$, and therefore it suffices to show that $[\tilde{f}_{i}, \tilde{f}_{j}]=1$ at some point of $X$.
Let us take a point $x_{0} \in X_{0}$.
Note that $f_{i}', f_{j}'$ are identity near $x_{0}$ by construction,
and hence $\tilde{f}_{i}, \tilde{f}_{j}$ are deck transformations of $P \to U$ over $x_{0}$.
Since the group of such deck transformations is $\Z/2$, which is abelian, we have that $[\tilde{f}_{i}, \tilde{f}_{j}]=1$ at $x_{0}$.

Now we have seen that $[\tilde{f}_{i}, \tilde{f}_{j}]=1$ for $i \neq j$, and therefore 
$\tilde{f_{1}}, \ldots, \tilde{f_{k}}$ induces a fiberwise spin structure on the mapping torus $E$.
Moreover, as mentioned above, $\tilde{f}_{i}, \tilde{f}_{j}$ are deck transformations of $P \to U$ over $X_{0}$.
If necessary, replacing $\tilde{f}_{i}$ with $-1 \cdot \tilde{f}_{i}$, where $-1$ denotes the unique non-trivial deck transformation of $P \to U$, we can assume that $\tilde{f}_{i}$ act trivially on $P|_{X_{0}}$, in particular on $\del X$.
Thus we obtain a fiberwise spin structure on $E$ whose restriction to the fiberwise boundary gives rise to the trivial bundle of spin $3$-manifolds.
This completes the proof of the assertion about topological spin structure.



The proof of the assertion about topological spin$^{c}$ structure is similar to the proof of \cite[Lemma~4.2]{KN20}.
As in the above spin case,  let $P_{i} \to X_{i}$ denote a principal $(\widetilde{GL}^{+}(n,\R) \times S^{1})/\pm1$-bundle corresponding to the smooth spin$^{c}$ structure $\fraks_{i}$, and $P$ denote a principal $Spin^{c}Top(n)$-bundle covering $U$ corresponding to the spin$^{c}$ structure $\fraks$.
Then it follows that $f_{i}'$ lifts to an $Spin^{c}Top(n)$-automorphism $\hat{f}_{i}$ of $P_{i}$.
Note that $\hat{f}_{i}|_{\del D_{i}}$ covers the identity.
Because of the fibration $S^{1} \to Spin^{c}Top(4) \to STop(4)$,
$\hat{f}_{i}|_{\del D_{i}}$ can be regarded as a continuous map $\del D_{i} \to S^{1}$.
Take a collar neighborhood $N(\del D_{i}) \cong [0,1] \times \del D_{i}$ of $\del D_{i}$ such that $1 \in [0,1]$ is the direction to the origin of $D_{i}$.
Because of $\pi_{3}(S^{1})=0$, one may find a continuous map $u_{i} : X \to S^{1}$ such that $u_{i}|_{\{0\} \times \del D_{i}} = \hat{f}_{i}|_{\{0\} \times \del D_{i}}$ and $u_{i}|_{X \setminus N(D_{i})} \equiv 1$.
Then $u_{i}^{-1} \cdot \hat{f}_{i}$ defines a lift $\tilde{f}_{i}$ of $f_{i}$ to an $Spin^{c}Top(n)$-automorphism of $P$, and since the support of $\tilde{f}_{i}$ lies in $X_{i}$,
$\tilde{f}_{i}$'s commute to each other and they are the identity over $X_{0} \supset Y$.
Therefore 
$\tilde{f_{1}}, \ldots, \tilde{f_{k}}$ induces a fiberwise spin$^{c}$ structure on the mapping torus $E$ whose restriction to the fiberwise boundary gives rise to the trivial bundle of spin$^{c}$ $3$-manifolds.
This completes the proof of the assertion about topological spin$^{c}$ structure.
\end{proof}

\subsection{Non-smoothable families of 4-manifolds with boundary}
\label{subsection: proof of main app}

In the following \cref{theo: main app}, non-smoothable families of 4-manifolds with boundary are detected using \cref{theo: main theo,theo: main theo2}.
Here let us clarify the word `non-smoothable family' in this paper:
we shall consider a continuous fiber bundle $E$ with fiber $4$-manifold $X$ with boundary.
If the structure group of $E$ reduces to $\Homeo(X,\del)$, but $E$ does not admit a reduction to $\Diff(X,\del)$, we say that $E$ is {\it non-smoothable}.
 
\begin{theo}
\label{theo: main app}
Let $Y$ be an oriented integral homology $3$-sphere.
Let $X$ be a simply-connected, compact, oriented, smooth, and indefinite $4$-manifold with boundary 
$Y$.
Suppose that $\sigma(X) \leq 0$.
Then:
\begin{enumerate}
\item 
Assume that at least one of the following holds:
\begin{enumerate}
\item $\sigma(X)<-8$ and $\delta(Y) \leq 0$.
\item $\delta(Y) < 0$, and in addition $\sigma(X)<0$ if $X$ is non-spin.
\item $\sigma(X) = -8$, $\delta(Y) = 0$ and $\mu(Y)=1$.
\end{enumerate}
Then there exists a non-smoothable $\Homeo(X,\del)$-bundle
\[
X \to E \to T^{b^{+}(X)}.
\]
\item Suppose that $X$ is spin.
\begin{enumerate}
\item If $-\sigma(X)/8 > \gamma(Y)$, there exists a non-smoothable $\Homeo(X,\del)$-bundle
\[
X \to E \to T^{b^{+}(X)}.
\]
\item If $b^{+}(X) >1$ and $-\sigma(X)/8 > \beta(Y)$, there exists a non-smoothable $\Homeo(X,\del)$-bundle
\[
X \to E \to T^{b^{+}(X)-1}.
\]
\item If $b^{+}(X) > 2$ and $-\sigma(X)/8 > \alpha(Y)$, there exists a non-smoothable $\Homeo(X,\del)$-bundle
\[
X \to E \to T^{b^{+}(X)-2}.
\] 
\end{enumerate}
\end{enumerate}
\end{theo}

In \cite{KKN19} Kato, Nakamura and the first author introduced an idea to detect non-smoothable families of closed $4$-manifold using families gauge theory and to apply them to extract difference between diffeomorphism groups and homeomorphism groups \cite[Theorem~1.4, Corollary~1.5]{KKN19}.
That was extensively generalized by Baraglia~\cite{Ba19} soon later.
\Cref{theo: main app} is an analog of \cite[Theorem~1.8]{Ba19}.

To prove \cref{theo: main app}, we need the following results regarding topological 4-manifolds with boundaries by Freedman.
Roughly speaking, these results state that Freedman's classification result holds also for topological $4$-manifolds with homology sphere boundary.

\begin{theo}[See, for example, \cite{B86,B93}]
\label{theo: Freedman boundary}
Let $Y$ be an integral homology $3$-sphere.
\begin{itemize} 
\item[(i)] 
The set of simply-connected compact topological 4-manifolds with boundary $Y$ having an even intersection form up to homeomorphism is determined by unimodular intersection forms up to isomorphism. 
\item[(ii)]
The set of simply-connected compact topological 4-manifolds with boundary $Y$ having an odd intersection form up to homeomorphism is determined by unimodular intersection forms and Kirby--Siebenmann invariant up to isomorphism. 
\end{itemize}
\end{theo}

\begin{theo}[{\cite[9.3C Corollary]{FQ90}}]
\label{theo: contractible bound}
Every integral homology $3$-sphere bounds a contractible topological $4$-manifold. 
\end{theo}

 Now we may start the proof of \cref{theo: main app}.
 A principal idea of the construction of non-smoothable families here is based on arguments for families of closed $4$-manifolds: \cite[Sections~3, 4]{Na10}, \cite[Theorem~4.1]{KKN19}, and \cite[Theorem 10.3]{Ba19}.
 
 \begin{proof}[Proof of \cref{theo: main app}]
By \cref{rem: comparison assumption},
if $X$ is spin and satisfies one of the assumptions in (1-a), (1-b), (1-c) of the statement of \cref{theo: main app}, then $X$ satisfies the assumption in (2-a).
Moreover the conclusion of the case (1) is just the same as (2-a).
Therefore, when we give a proof of the case (1), we can additionally suppose that $X$ is not spin, since the case that $X$ is spin is deduced from the case (2-a), which will be proven independently from the case (1).
 
  Let $-E_{8}$ denote the negative-definite $E_{8}$-manifold.
 Let $W$ be a contractible topological $4$-manifold bounded by $Y$, whose existence is ensured by 
 \cref{theo: contractible bound}.
Here we note the equality
 \begin{align}
 \label{eq: ks mu}
\ks(W) = \mu(Y) 
 \end{align}
 in $\Z/2$, where $\ks(W) \in H^{4}(W,\del W;\Z/2) \cong \Z/2$ denote the Kirby--Siebenmann invariant of $W$.
This equality is checked as follows.
We refer the readers to \cite[Subsection~8.2]{FNOP19} for some properties of the Kirby--Siebenmann invariant.
Recall that both of the signature and the Kirby--Siebenmann invariant are additive under the sum operation along a codimension-$0$ submanifold.
Recall also that the Kirby--Siebenmann invariant vanishes for a smooth $4$-manifold, even for the case with boundary.
Lastly, recall that we have a formula $\ks(Z) = \sigma(Z)/8 \mod 2$ for an oriented spin closed $4$-manifold $Z$.
By the definition of the Rohlin invariant,
we may take an oriented spin compact smooth $4$-manifold $W'$ bounded by $Y$ with $\mu(Y)=\sigma(W')/8 \mod 2$.
Setting $Z = -W' \cup_{Y} W$, we have
\begin{align*}
\ks(W)
&=\ks(-W') + \ks(W) \\
&= \ks(Z) = \sigma(Z)/8 
= -\sigma(W')/8 + \sigma(W)/8
= \mu(Y) + \sigma(W)/8.
\end{align*}
In particular, since $W$ was taken to be contractible, we have \eqref{eq: ks mu}.

First let us suppose that $X$ is not spin, and suppose that $\sigma(X)<-8$ and $\delta(Y) \leq 0$.
Let $-\CP^{2}_{\mathrm{fake}}$ denote the fake $\CP^{2}$, namely the closed simply-connected topological $4$-manifold whose intersection form is $(-1)$ and has non-zero Kirby--Siebenmann invariant.
It follows from \eqref{eq: ks mu} and \cref{theo: Freedman boundary} that $X$ is homeomorphic to
\begin{align}
\label{eq: nonspin connected sum ho}
b^{+}(X)(S^{2} \times S^{2}) \# n(-\CP^{2}) \# (-E_{8}) \#(-\CP^{2}_{\mathrm{fake}})\#W
\end{align}
or
\begin{align}
\label{eq: nonspin connected sum}
b^{+}(X)(S^{2} \times S^{2}) \# n(-\CP^{2}) \# (-E_{8}) \#W
\end{align}
for some $n \geq 0$, according to $\mu(Y)=0$ or $\mu(Y)=1$.
Let $f_{0} \in \Diff(S^2\times S^2)$ be an orientation-preserving self-diffeomorphism satisfying the following two properties:
\begin{itemize}
\item There exists an embedded $4$-disk in $S^2\times S^2$ such that the restriction of $f_{0}$ on the disk is the identity map.
\item $f_{0}$ reverses orientation of $H^+(S^2\times S^2)$.
\end{itemize}
Such $f_{0}$ can be made by starting with the componentwise complex conjugation of $\CP^{1} \times \CP^{1} = S^{2} \times S^{2}$, and deforming it around a fixed point by isotopy so that it has a fixed disk.
Let $f_{1}, \ldots, f_{b^{+}}$ be copies of $f_{0}$ on the connected sum factors of $b^{+}(X)(S^{2} \times S^{2})$.
Since $f_{i}$ have fixed disks, we can extend them as homeomorphisms of $X$ by the identity map on the other connected sum factors in \eqref{eq: nonspin connected sum ho} or \eqref{eq: nonspin connected sum}.
Since the extended homeomorphisms obviously mutually commute, they give rise to the multiple mapping torus
\[
X \to E \to T^{b^{+}}.
\]
Note that the restrictions of $f_{1}, \ldots, f_{b^{+}}$ onto the boundary are the identity maps, 
and hence $E$ is a $\Homeo(X,\del)$-bundle.
Since $f_{0}$ was taken to reverse orientation of $H^+(S^2\times S^2)$,
it is easy to see that the associated bundle $H^{+}(E) \to T^{b^{+}}$ satisfies $w_{b^{+}}(H^{+}(E)) \neq 0$.

Let $c \in H^{2}(X;\Z)$ be a cohomology class given by
\[
c = (0,e_{1}, \ldots, e_{n}, 0, e)
\]
or
\[
c = (0,e_{1}, \ldots, e_{n}, 0)
\]
under the direct sum decomposition
\[
H^{2}(X) = H^{2}(b^{+}(X)(S^{2} \times S^{2})) \oplus H^{2}(-\CP^{2})^{\oplus n} \oplus H^{2}(-E_{8}) \oplus H^{2}(-\CP^{2}_{\mathrm{fake}}),
\]
or
\[
H^{2}(X) = H^{2}(b^{+}(X)(S^{2} \times S^{2})) \oplus H^{2}(-\CP^{2})^{\oplus n} \oplus H^{2}(-E_{8}),
\]
according to $\mu(Y)=0$ or $\mu(Y)=1$.
Here $e_{i}$ and $e$ are a generator of $H^{2}(-\CP^{2})$ and a generator of $H^{2}(-\CP^{2}_{\mathrm{fake}})$ respectively.
By \cref{lem: lift top spin}, $E$ admits a fiberwise topological spin$^{c}$ structure whose characteristic restricted over a fiber coincides with $c$ above.

Now suppose that $E$ is smoothable, namely $E$ reduces to a $\Diff(X,\del)$-bundle.
Then the topological spin$^{c}$ structure of $E$ above induces a smooth spin$^{c}$ structure, and the restriction of the spin$^{c}$ structure over a fiber, denoted by $\fraks$, has $c_{1}(\fraks)=c$.
Now we have $(c_{1}(\fraks)^{2}-\sigma(X))/8=1$, and hence \cref{theo: main theo} implies that $1 \leq \delta(Y)$.
This contradicts the assumption that $\delta(Y) \leq 0$, and hence $E$ is not smoothable.

Next, let us suppose that $X$ is not spin, and suppose that $\sigma(X) = -8$, $\delta(Y) = 0$, and $\mu(Y)=1$.
By \eqref{eq: ks mu} we have that $\ks(W)=1$ in this case.
It follows from \cref{theo: Freedman boundary} that $X$ is homeomorphic to \eqref{eq: nonspin connected sum}
for $n = 0$.
The remaining argument is exactly the same as the previous paragraph.

Next, let us suppose that $X$ is not spin, and suppose that $\delta(Y) < 0$ and $\sigma(X)<0$.
Most of arguments here are just the same as the arguments until previous paragraph.
First, it follows from \cref{theo: Freedman boundary} that $X$ is homeomorphic to
\begin{align*}
b^{+}(X)(S^{2} \times S^{2}) \# (b^{-}(X)-b^{+}(X))(-\CP^{2}) \#W
\end{align*}
if $\mu(Y)=0$, and $X$ is homeomorphic to
\begin{align*}
b^{+}(X)(S^{2} \times S^{2}) \# (b^{-}(X)-b^{+}(X)-1)(-\CP^{2}) \# (-\CP^{2}_{\mathrm{fake}})\#W
\end{align*}
if $\mu(Y)=1$ respectively.
Let $f_{1}, \ldots, f_{b^{+}}$ be copies of $f_{0}$ on the connected sum factors of $b^{+}(X)(S^{2} \times S^{2})$.
Extending them as homeomorphisms of $X$, we obtain a $\Homeo(X,\del)$-bundle $X \to E \to T^{b^{+}}$ for which $w_{b^{+}}(H^{+}(E)) \neq 0$.
Let us take $c \in H^{2}(X;\Z)$ defined by
\[
c = (0, e_{1}, \ldots, e_{b^{-}-b^{+}}) 
\]
or
\[
c = (0, e_{1}, \ldots, e_{b^{-}-b^{+}-1}, e) 
\]
under
\[
H^{2}(X) = H^{2}(b^{+}(X)(S^{2} \times S^{2})) \oplus H^{2}(-\CP^{2})^{\oplus b^{-}-b^{+}},
\]
or
\[
H^{2}(X) = H^{2}(b^{+}(X)(S^{2} \times S^{2})) \oplus H^{2}(-\CP^{2})^{\oplus b^{-}-b^{+}-1} \oplus H^{2}(-\CP^{2}_{\mathrm{fake}}),
\]
according to $\mu(Y)=0$ or $\mu(Y)=1$.
Then we have $c^{2} -\sigma(X)=0$.
Arguing exactly as in the last case, if we suppose that $E$ is smoothable, \cref{theo: main theo} implies that $0 \leq \delta(Y)$.
This contradicts the assumption that $\delta(Y) \leq -1$, and hence $E$ is not smoothable.

Next, let us suppose that $X$ is spin, and suppose that $-\sigma(X)/8 > \gamma(Y)$.
By the definition of the Rohlin invariant , we have $\sigma(X)/8 \equiv \mu(Y) \mod 2$,
and hence $\ks(W) = \sigma(X)/8 \mod 2$ holds by \eqref{eq: ks mu}.
Then it follows from \cref{theo: Freedman boundary} that $X$ is homeomorphic to
\begin{align*}
b^{+}(X)(S^{2} \times S^{2})\# n(-E_{8}) \#W
\end{align*}
for some $n \geq 0$.
As well as the non-spin case, considering copies of $f_{0}$ on the connected sum factors of $b^{+}(X)(S^{2} \times S^{2})$ and extend them to the whole of $X$ as homeomorphisms, we obtain 
a $\Homeo(X,\del)$-bundle $X \to E \to T^{b^{+}}$ for which $w_{b^{+}}(H^{+}(E)) \neq 0$.
By \cref{lem: lift top spin}, $E$ admits a fiberwise topological spin structure.
Arguing exactly as in the non-spin case, if we suppose that $E$ is smoothable, \cref{theo: main theo2} implies that $-\sigma(X)/8 \leq \gamma(Y)$.
This contradicts the assumption that $-\sigma(X)/8 > \gamma(Y)$, and hence $E$ is not smoothable.

The remaining cases, $X$ is spin and $b^{+}(X) > 1, -\sigma(X)/8 > \beta(Y)$, or $b^{+}(X) >2, -\sigma(X)/8 > \alpha(Y)$, are also similar.
Consider copies $f_{1}, \ldots, f_{b^{+}-1}$ or $f_{1}, \ldots, f_{b^{+}-2}$ of $f_{0}$ on the connected sum factors of $(b^{+}-1)(S^{2} \times S^{2})$ or $(b^{+}-2)(S^{2} \times S^{2})$ inside $b^{+}(X)(S^{2} \times S^{2})$, according to the assumption on $\beta(Y)$ or $\alpha(Y)$.
Then we obtain $X \to E \to T^{b^{+}-1}$ or $X \to E \to T^{b^{+}-2}$ for which $w_{b^{+}-1}(H^{+}(E)) \neq 0$ or $w_{b^{+}-2}(H^{+}(E)) \neq 0$ respectively.
\Cref{theo: main theo2} implies that this $E$ is not smoothable.
This completes the proof of \cref{theo: main app}.
\end{proof}

\subsection{Comparison between $\Diff$ and $\Homeo$} 
  \label{subsec: Comparison between Diff and Homeo}

Let us extract homotopical difference between various diffeomorphism groups and homeomorphism groups using \cref{theo: main app}.
The results in this \lcnamecref{subsec: Comparison between Diff and Homeo} contain \cref{cor: Diff Homeo absolute,cor: rel Diff Homeo} stated in the introduction.
First let us start with comparison between the relative diffeomorphism and homeomorphism groups:

\begin{cor}
\label{cor: rel Diff Homeo ap}
Let $Y$ be an oriented integral homology $3$-sphere.
Let $X$ be a simply-connected, compact, oriented, smooth, and indefinite $4$-manifold with boundary 
$Y$.
Suppose that $\sigma(X) \leq 0$.
Suppose that $X$ and $Y$ satisfy at least one of the following conditions:
\begin{enumerate}
\item $\sigma(X)<-8$ and $\delta(Y) \leq 0$.
\item $\delta(Y) < 0$, and in addition $\sigma(X)<0$ if $X$ is non-spin.
\item $\sigma(X) = -8$, $\delta(Y) = 0$ and $\mu(Y)=1$.
\item $X$ is spin and $-\sigma(X)/8 > \gamma(Y)$.
\item $X$ is spin, $b^{+}(X) > 1$ and $-\sigma(X)/8 > \beta(Y)$.
\item $X$ is spin, $b^{+}(X) > 2$ and $-\sigma(X)/8 > \alpha(Y)$.
\end{enumerate}
Then the inclusion map
\[
\Diff(X,\del) \inc \Homeo(X,\del)
\]
is not a weak homotopy equivalence.

More precisely:
\begin{itemize}
\item If at least one of (1), (2), (3), (4) is satisfied, the induced map
\[
\pi_{n}(\Diff(X,\del)) \to \pi_{n}(\Homeo(X,\del))
\]
is not an isomorphism for some $n \in \{0, \ldots, b^{+}(X)-1\}$.
\item If (5) is satisfied, the induced map
\[
\pi_{n}(\Diff(X,\del)) \to \pi_{n}(\Homeo(X,\del))
\]
is not an isomorphism for some $n \in \{0, \ldots, b^{+}(X)-2\}$.
\item If (6) is satisfied, the induced map
\[
\pi_{n}(\Diff(X,\del)) \to \pi_{n}(\Homeo(X,\del))
\]
is not an isomorphism for some $n \in \{0, \ldots, b^{+}(X)-3\}$.
\end{itemize}
\end{cor}

\begin{proof}
This follows from \cref{theo: main app} combined with the standard obstruction theory, as well as the proof of \cite[Corollary~10.5]{Ba19}.
\end{proof}

\begin{cor}
\label{cor: Diff Homeo absolute ap}
Let $Y$ be an oriented integral homology $3$-sphere.
Let $X$ be a simply-connected, compact, oriented, smooth, and indefinite $4$-manifold with boundary $Y$.
Suppose that $\sigma(X) \leq 0$.
Suppose that $X$ and $Y$ satisfy at least one of the conditions (1)-(6) in the statement of \cref{cor: rel Diff Homeo ap}.
Then the inclusion map
\[
\Diff(X) \inc \Homeo(X)
\]
is not a weak homotopy equivalence.

More precisely:
\begin{itemize}
\item If at least one of (1), (2), (3), (4) is satisfied, the induced map
\[
\pi_{n}(\Diff(X)) \to \pi_{n}(\Homeo(X))
\]
is not an isomorphism for some $n \in \{0, \ldots, b^{+}(X)\}$.
\item If (5) is satisfied, the induced map
\[
\pi_{n}(\Diff(X)) \to \pi_{n}(\Homeo(X))
\]
is not an isomorphism for some $n \in \{0, \ldots, b^{+}(X)-1\}$.
\item If (6) is satisfied, the induced map
\[
\pi_{n}(\Diff(X)) \to \pi_{n}(\Homeo(X))
\]
is not an isomorphism for some $n \in \{0, \ldots, b^{+}(X)-2\}$.
\end{itemize}
\end{cor}

\begin{proof}
Recall that, for an arbitrary orientable closed smooth $3$-manifold, the inclusion map from the diffeomorphism group into the homeomorphism group is a weak homotopy equivalence.
(This is a result by Cerf~\cite{Ce68}, combined with the solution to the Smale conjecture by Hatcher~\cite{Ha83}.
See~\cite{Hat80}.)

As noted by Pardon~\cite[Subsection~2.1]{Pa19},
we have an exact sequence
\begin{align}
\label{eq: exact Diff}
1 \to \Diff(X,\del) \to \Diff(X) \to \Diff(Y),
\end{align}
where the image of the last map is a union of connected components.
Similarly we have
\begin{align}
\label{eq: exact Homeo}
1 \to \Homeo(X,\del) \to \Homeo(X) \to \Homeo(Y).
\end{align}
These exact sequences induce long exact sequences of homotopy groups, although the final maps on $\pi_{0}$ may not be surjections.
A natural termwise inclusion from \eqref{eq: exact Diff} to \eqref{eq: exact Homeo} gives rise to a commutative diagram between two long exact sequences.
Now we can deduce from the fact in dimension $3$ explained in the last paragraph and \cref{cor: rel Diff Homeo ap} that $\Diff(X) \inc \Homeo(X)$ is not a weak homotopy equivalence  with the assistance of the five lemma.
More precisely, we may obtain estimates for $n$ from those in \cref{cor: rel Diff Homeo ap}, but note that the new estimates may be weaker than those in \cref{cor: rel Diff Homeo ap} at most $1$.
\end{proof}

Here we give the proof of \cref{theo: not surj on alg aut}:

\begin{proof}[Proof of \cref{theo: not surj on alg aut}]
\Cref{theo: main app} implies that there exists a non-smoothable $\Homeo(X,\del)$-bundle
 $X \to E \to S^{1}$.
 This implies that \eqref{eq: pizero del} is not a surjection.
 The remaining statement follows from this and the fact that $\Diff(Y) \inc \Homeo(Y)$ is a weak homotopy equivalence, with the assistance of the four lemma as well as the proof of \cref{cor: Diff Homeo absolute ap}.
 (Note that, however, we can detect exactly $\pi_{0}$ not like the statement for $\Diff(X) \inc \Homeo(X)$ in \cref{cor: Diff Homeo absolute ap}.)
\end{proof}

\begin{rem}
\label{rem: non surj closed}
It would be interesting to compare \cref{theo: not surj on alg aut} with the situation for closed $4$-manifolds.
For a closed smooth $4$-manifold $X$,
\begin{align}
\label{eq: pi zero closed 1}
\pi_{0}(\Diff(X)) \to \pi_{0}(\Homeo(X))
\end{align}
is often a surjection by Wall's theorem~\cite{Wa64} on the realizability of elements of $\Aut(H^{2}(X;\Z))$ by diffeomorphisms and Quinn's theorem~\cite{Qu86}, which shows that
\[
\pi_{0}(\Homeo^{+}(X)) \to \Aut(H^{2}(X;\Z))
\]
is an isomorphism as far as $X$ is simply-connected.
There are few known examples of closed smooth $4$-manifolds $X$ for which \eqref{eq: pi zero closed 1} are not surjections: the first example is a $4$-manifold homeomorphic to $\CP^{2}\#n(-\CP^{2})$ for $n>9$ by Friedman--Morgan~\cite{FM88}.
A $K3$ surface is also an example by a result by Donaldson~\cite{Do90}, and in fact so is every homotopy $K3$ surface, which one can check using a result by Morgan and Szab{\'o}~\cite{MS97}.
It follows from Baraglia's constraint~\cite[Theorem 1.1]{Ba19} that a $4$-manifold homeomorphic to $\CP^{2} \# n(-\CP^{2})$ with $n>9$ or an Enriques surface is also an example, and so is a stabilization of such a $4$-manifold by the connected sum with some non-simply-connected $4$-manifolds by Nakamura and the first author \cite[Corollary~1.6]{KN20}.
\end{rem}
 
\subsection{Examples} 
  \label{subsec: Other applications and examples}
In this subsection we give examples for  the main applications \cref{cor: rel Diff Homeo ap,cor: Diff Homeo absolute ap} as \cref{ex: first ex}, and examples for \cref{theo: not surj on alg aut}, which is a specialization of the main applications for small $b^{+}$, as \cref{ex: first first ex,ex: Ex of non surj1.5,ex: Ex of non surj2 rev,ex: Ex of non surj2,ex: Ex of non surj3}.
In \cref{ex: Ex of non surj2 rev} we give an example which is detected by the invariant $\gamma$, but cannot be detected by the Fr{\o}yshov invariant $\delta$.
In \cref{ex: Ex of non surj2} we give an example detected by the invariants $\beta$, not by $\gamma$ or $\delta$.
In \cref{ex: Ex of non surj3} we give an example detected by the invariants $\alpha$, not by $\beta$, $\gamma$ or $\delta$.

 First let us consider examples for \cref{cor: rel Diff Homeo ap,cor: Diff Homeo absolute ap}.
As mentioned in \cref{rem: huge ex}, we can easily find a huge number of examples of $(X,Y)$ to which these main applications can apply: just find $(X,Y)$ with $\sigma(X)<-8$ and $\delta(Y) \leq 0$. 
Specializing to the case that $\delta(Y)=0$, this can be regarded as an analog of the assumption $| \sigma (X) | > 8$ of Baraglia's constraint \cite[Corollary 1.9]{Ba19} for closed $4$-manifolds.
However, in our situation, we may obtain examples of $(X,Y)$ with $|\sigma(X)| \leq 8$ thanks to the assistance of the Fr{\o}yshov invariant.
Let us note such an example:
 \begin{ex}
\label{ex: first first ex}
 Let $n \geq 1$ and set $Y = -\Sigma(2,3,12n-1)$.
Let $X$ be an oriented spin compact simply-connected smooth $4$-manifold bounded by $Y$ with the intersection form
\begin{align*}
 \left(
    \begin{array}{cc}
     0  & 1  \\
      1 & 0  \\
    \end{array}
  \right).
\end{align*}
An example of such $X$ is the nucleus $N(2n)$ inside the elliptic surface $E(2n)$
(see, for example, \cite[Subsection~5.3]{Ma14}), and one may take also exotic nuclei as examples. 
 Since we have $\delta(Y)=-1$, the pair $(X,Y)$ satisfies the assumption (2) of \cref{theo: not surj on alg aut}, and thus we have that
\[
\pi_0(\operatorname{Diff}  (X , \partial )) \to   \pi_0(\operatorname{Homeo}  (X , \partial ))
\]
and
\[
\pi_0(\operatorname{Diff}(X)) \to   \pi_0(\operatorname{Homeo}  (X))
\]
are not surjections. 
\end{ex}

\begin{rem}
\label{rem: first ex}
To a $4$-manifold $X'$ obtained as the boundary connected sum of $X$ in \cref{ex: first first ex} with any contractible $4$-manifold $W$ with integral homology $3$-sphere boundary, we may still apply \cref{theo: not surj on alg aut} and conclude that 
\[
\pi_0(\operatorname{Diff}  (X' , \partial )) \to   \pi_0(\operatorname{Homeo}  (X' , \partial ))
\]
and
\[
\pi_0(\operatorname{Diff}(X')) \to   \pi_0(\operatorname{Homeo}  (X'))
\]
are not surjections.
This is because we have $\delta(\del W)=0$ and $\delta$ is additive under connected sum.
Such a remark applies also to many of examples below.
\end{rem}
 
Let us give a remark on comparisons between various Fr{\o}yshov-type invariants.
The authors were informed by Ciprian Manolescu of the content of this remark.

\begin{rem}
\label{rem: comparison between Froyshov type invariants}
The following fact is pointed out in \cite[Remark 1.1]{LRS18}. 
In the work of Kutluhan, Lee, and Taubes \cite{KLTI}, \cite{KLTII}, \cite{KLTIII}, \cite{KLTIV}, \cite{KLTV}, alternatively, the work of Colin, Ghiggini, and Honda \cite{CGHI} \cite{CGHII} \cite{CGHIII}  and Taubes \cite{Ta10},  it is proved that the monopole Floer homology and the Heegaard Floer homology in coefficients $\Z$ are isomorphic to each other.
 In particular, with $\F$-coefficients, we also have an isomorphism between the monopole Floer homology and the Heegaard Floer homology. Moreover, the $\Q$-gradings are compared in \cite{RG18}, \cite{CD13} and \cite{HR17}.  This proves 
\[
\frac{1}{2} d ( Y, \mathfrak{t}, \F  ) =   -h ( Y, \mathfrak{t},  \F   ), 
\]
where $d ( Y, \mathfrak{t}, \F  )$ is the correction term of Heegaard Floer homology defined over $\F$-coefficient and $h ( Y, \mathfrak{t},  \F   )$ is the monopole Fr\o yshov invariant defined over $\F$-coefficient.

On the other hand, in \cite{LM18}, Lidman and Manolescu gave a grading preserving isomorphism between the $S^1$-equivariant cohomology of $SWF(Y, \mathfrak{t})$ and the monopole Floer homology over $\Z$. This proves 
\[
-h ( Y, \mathfrak{t},  \F   ) = \delta (Y, \mathfrak{t}). 
\]
Summarizing the results above, we have 
\begin{align*}
\frac{1}{2} d ( Y, \mathfrak{t}, \F  ) =  \delta (Y, \mathfrak{t}).
\end{align*}
The equality enables us to calculate the invariant $\delta$ by a combinatorial way. 
\end{rem}
  
  Next we provide another family of examples satisfying the assumption of \cref{theo: not surj on alg aut} coming from surgeries of knots in $S^3$. 
\begin{ex}
\label{ex: Ex of non surj}
Let $K$ be any knot in $S^3$. Since the $(+1)$-surgery $S^3_1 (K)$ of $K$ admits a positive-definite bounding $W_1(K)$ as the trace of the $(+1)$-surgery on $K$, we always have $\delta (S_1 (K) ) \leq 0$. 
We suppose that 
\begin{align}
\label{eq: ex neg sur}
\delta (S_1 (K) ) < -1,
\end{align}
where we shall give concrete examples of such $K$ below.
We define a pair $(W_K,  Y_K)$ as the boundary connected sum of $(W_1 (K), S_1 (K))$ and a simply-connected $(-E_8)$-bounding of $\Sigma(2,3,5)$.
Note that $b^+ (W_K) =1$, $\sigma(W_{K}) < 0$, and $W_K$ is simply-connected and the intersection form of $W_K$ is indefinite. Therefore the pair $(W_K,  Y_K )$ satisfies the assumption (2) of \cref{theo: not surj on alg aut}, and thus we have that
\[
\pi_0(\operatorname{Diff}  (W_K , \partial )) \to   \pi_0(\operatorname{Homeo}  (W_K , \partial ))
\]
and
\[
\pi_0(\operatorname{Diff}  (W_K)) \to   \pi_0(\operatorname{Homeo}  (W_K))
\]
are not surjections. 

In order to find a concrete family of examples of $K$ with \eqref{eq: ex neg sur}, 
we consider
\[
K = T(2,2n-1) 
\]
for any positive integer $n$, where $T(p,q)$ denotes the $(p,q)$-torus knot.
It is mentioned in \cite{Fr04} that $-S^3_1 (T(2,2n-1) )= \Sigma(2,2n-2, 4n-3)$ has 
\[
\Gamma_{4n} = \Set{ \sum_{1 \leq i \leq 4n}  x_i e_i \in \R^{4n} | \sum x_i \in 2\Z, \ 2x_i \in \Z, x_i -x_j \in \Z }
\]
as the negative-definite intersection from of the minimal resolution $W_{4n}$, where $\{e_i\}$ is an orthonormal basis of $\R^{4n}$. 
Then by using an inequality by Fr{\o}yshov~\cite{Fr96} for $W_{4n}$, which is the same as \cref{theo: main theo} for $B=\{\pt\}$, we obtain a family of estimates 
\[
 \left\lfloor \frac{n}{2} \right\rfloor \leq  \delta(\Sigma(2,2n-1, 4n-3) ). 
 \]
 This proves 
 \[
  \delta(S^3_1 (T(2,2n-1) ) \leq -   \left\lfloor \frac{n}{2} \right\rfloor 
  \]

 We can see that for any positive integer $n \geq 4$, 
 \[
 T(2,2n-1) 
 \]
 satisfies \eqref{eq: ex neg sur}.  
\end{ex}

Let us also give an example which is detected by the condition (3) of \cref{theo: not surj on alg aut}:
  
\begin{ex}
\label{ex: Ex of non surj1.5}
For $n \geq 1$, set $Y=-\Sigma(2,3,12n-5)$.
It is known that $\delta(Y)=0$ and $\mu(Y)=1$.
(See, for example, \cite[Subsection~3.8]{Ma16}.)
Note that $Y$ bounds an oriented compact simply-connected smooth spin $4$-manifold $X$ having the following intersection form (see, for example, \cite[Subsection~5.3]{Ma14}):
\begin{align*}
 (-E_8)  \oplus 
 \left(
    \begin{array}{cc}
     0  & 1  \\
      1 & 0  \\
    \end{array}
  \right),
\end{align*}
We have $b^{+}(X)=1$ and $-\sigma(X)/8=1$.
Hence $(X,Y)$ satisfies the assumption (3) of \cref{theo: not surj on alg aut}, and thus
we have that
\[
\pi_0(\operatorname{Diff}  (X , \partial )) \to   \pi_0(\operatorname{Homeo}  (X , \partial ))
\]
and
\[
\pi_0(\operatorname{Diff}  (X)) \to   \pi_0(\operatorname{Homeo}  (X))
\]
are not surjections.
%
\end{ex}

Next let us give an example which is detected by the invariant $\gamma$, but cannot be detected by the Fr{\o}yshov invariant $\delta$.

\begin{ex}
\label{ex: Ex of non surj2 rev}
Note that $2\Sigma(2,311)$ bounds an oriented compact simply-connected smooth spin $4$-manifold $W$ with the intersection form
\begin{align*}
2(-E_{8}) \oplus 
 \left(
    \begin{array}{cc}
     0  & 1  \\
      1 & 0  \\
    \end{array}
  \right).
\end{align*}
An example of such $W$ can be found as a comdimension-0 submanifold with boundary of a $K3$ surface.
Indeed, a $K3$ surface contains three disjoint nuclei $\bigsqcup_{3} N(2)$, and the boundary is given by $\bigsqcup_{3} (-\Sigma(2,3,11))$ \cite{GM93}.
Remove two of the three nuclei from $K3$, and take an inner connected sum of the two boundary components of $K3 \setminus \bigsqcup_{2} {\rm Int}N(2)$,
then we get an example of such $W$.

Let $Y=2\Sigma(2,3,11)\#n\Sigma(2,3,5)$ for $n \geq 0$, and let $X$ be the boundary connected sum of $W$ with a simply-connected $n(-E_{8})$-bounding of $n\Sigma(2,3,5)$.
Obviously $X$ is spin, $b^{+}(X)=1$ and $-\sigma(X)/8=n+2$.
On the other hand, as computed by Manolescu~\cite[Subsection~3.8]{Ma16}, we have
\[
\beta(\Sigma(2,3,11))=0, \quad \alpha(\Sigma(2,3,5))=1.
\]
It follows from the connected sum formulae on $\alpha, \beta, \gamma$ by Stoffregen~\cite[Theorem 1.1]{Sto172} that 
\[
\gamma(Y) 
\leq \gamma(2\Sigma(2,3,11)) + n\alpha(\Sigma(2,3,5))
\leq 2\beta(\Sigma(2,3,11)) + n\alpha(\Sigma(2,3,5)) =n,
\]
hence the assumption (4) of \cref{theo: not surj on alg aut} is satisfied for $X$ and $Y$.
Thus
we have that
\[
\pi_0(\operatorname{Diff}  (X , \partial )) \to   \pi_0(\operatorname{Homeo}  (X , \partial ))
\]
and
\[
\pi_0(\operatorname{Diff}  (X)) \to   \pi_0(\operatorname{Homeo}  (X))
\]
are not surjections.

It is worth noting that this example cannot be detected by $\delta$:
because of $\delta(\Sigma(2,3,7))=0$, we have $\delta(Y)=n$, and hence $-\sigma(X)/8=\delta(Y)$.
\end{ex}

\begin{rem}
In \cref{ex: Ex of non surj2 rev},
if we take $n=0$ and $W$ as a codimension-0 submanifold with boundary of $K3$, 
the result on $\pi_{0}$ can be deduced from a classical theorem by Donaldson~\cite{Do90} regarding a closed $4$-manifold:
 a $K3$ surface does not admit a diffeomorphism which reverses orientation of $H^{+}(K3)$.
 
 
 The same remark applies also to the following \cref{ex: Ex of non surj2,ex: Ex of non surj3}.
 \end{rem}

Next let us give an example detected by the invariant $\beta$.
  
\begin{ex}
\label{ex: Ex of non surj2}
Note that $\Sigma(2,3,11)$ bounds an oriented compact simply-connected smooth spin $4$-manifold $W$ with the intersection form
\begin{align*}
 2(-E_8)  \oplus 
 2\left(
    \begin{array}{cc}
     0  & 1  \\
      1 & 0  \\
    \end{array}
  \right).
\end{align*}
For example, such $W$ can be obtained as the complement of $N(2)$ in a $K3$ surface.
 (See, for example, \cite[Subsection~5.3]{Ma14}.)
Let $Y=\Sigma(2,3,11)\#n\Sigma(2,3,5)$ for $n \geq 1$, and let $X$ be the boundary connected sum of $W$ with a simply-connected $n(-E_{8})$-bounding of $n\Sigma(2,3,5)$.
Obviously $X$ is spin, $b^{+}(X)=2$ and $-\sigma(X)/8=n+2$.
On the other hand, as computed by Manolescu~\cite[Subsection~3.8]{Ma16}, we have
\[
\beta(\Sigma(2,3,11))=0, \quad \alpha(\Sigma(2,3,5))=1.
\]
It follows from the connected sum formula by Stoffregen~\cite[Theorem 1.1]{Sto172} that 
\[
\beta(Y) \leq \beta(\Sigma(2,3,11)) + n\alpha(\Sigma(2,3,5)) =n,
\]
hence the assumption (5) of \cref{theo: not surj on alg aut} is satisfied for $X$ and $Y$.
Thus we have that
\[
\pi_0(\operatorname{Diff}  (X , \partial )) \to   \pi_0(\operatorname{Homeo}  (X , \partial ))
\]
and
\[
\pi_0(\operatorname{Diff}  (X)) \to   \pi_0(\operatorname{Homeo}  (X))
\]
are not surjections.

Note that $(X,Y)$ satisfies also (4) of \cref{cor: rel Diff Homeo ap}, indeed $\gamma(Y) \leq \gamma(\Sigma(2,3,11)) + n\alpha(\Sigma(2,3,5))=n$.
However it tells us only weaker information than the above result on $\pi_{0}$ detected by $\beta$: one can say only that $\pi_n(\operatorname{Diff}  (X , \partial )) \to   \pi_n(\operatorname{Homeo}  (X , \partial ))$ is not an isomorphism for at least one of $n \in \{0,1\}$, and that $\pi_n(\operatorname{Diff}  (X)) \to   \pi_n(\operatorname{Homeo}  (X))$ is not an isomorphism for at least one of $n \in \{0,1,2\}$.
\end{ex}

Let us give an example detected by the invariant $\alpha$.
  
\begin{ex}
\label{ex: Ex of non surj3}
Let $X$ be the interior connected sum of a homotopy $K3$ surface with a simply-connected $n(-E_{8})$-bounding of $n\Sigma(2,3,5)$.
Then $Y = n\Sigma(2,3,5)$ is the boundary of $X$.
Obviously $X$ is spin, $b^{+}(X)=3$ and $-\sigma(X)/8=n+3$.
On the other hand, as noted in \cref{ex: Ex of non surj2}, we have $\alpha(\Sigma(2,3,5))=1$.
It follows from the connected sum formula by Stoffregen~\cite[Theorem 1.1]{Sto172} that 
\[
\alpha(Y) \leq n\alpha(\Sigma(2,3,5)) =n,
\]
hence the assumption (6) of \cref{theo: not surj on alg aut} is satisfied for $X$ and $Y$.
Thus we have that
\[
\pi_0(\operatorname{Diff}  (X , \partial )) \to   \pi_0(\operatorname{Homeo}  (X , \partial ))
\]
and
\[
\pi_0(\operatorname{Diff}  (X)) \to   \pi_0(\operatorname{Homeo}  (X))
\]
are not surjections.
As in \cref{ex: Ex of non surj2}, 
$(X,Y)$ satisfies also (4) and (5) of \cref{cor: rel Diff Homeo ap}, but it tells us only weaker information than the above result on $\pi_{0}$ detected by $\alpha$.
\end{ex}

At the end of this section, we use spin boundings constructed by Saveliev\cite{Sa98} : 
 \begin{ex}
\label{ex: first ex}
 We consider the Brieskorn homology 3-sphere $\Sigma(p,q,r)$ for a pairwise relatively prime triple of positive integers $  (p ,q, r)$.
 Since $\Sigma(2,3,5)$ admits a positive scalar curvature, one can see that 
 \[
 \delta ( \Sigma(2,3,5) ) =1 . 
 \]
 On the other hand, for an odd positive integer $k$ and an odd positive integer $q$ with $q \equiv 3 \mod 4$, in \cite{Sa98} Saveliev constructed a family of simply connected spin boundings $W'_{q,k}$ of $-\Sigma (2,q, 2qk+1) $ whose intersection forms are isomorphic to 
 \[
 \left(\frac{q+1}{4} \right) (-E_8)  \oplus \left(
    \begin{array}{cc}
     0  & 1  \\
      1 & 0  \\
    \end{array}
  \right).
\]
Set 
\[
Y_{ k} :=  ( - \Sigma(2,3,5) ) \# (-\Sigma (2,3, 6k+1) ) . 
\]
Since $\delta (-\Sigma (2,3, 6k+1) ) =0$, we have $\delta (Y_{ k}) = -1$. 
Then we consider the boundary connected sum, denoted by $W_k$, of a simply-connected $E_8$-bounding of $- \Sigma(2,3,5)$ with $W'_{3,k}$. Note that the intersection form of $W_k$ is isomorphic to 
\[
9 \left(
    \begin{array}{cc}
     0  & 1  \\
      1 & 0  \\
    \end{array}
  \right), 
  \]
and hence $\sigma(W_k)=0$. 
  Moreover, $ W_k$ is spin, simply-connected and 
  \[
 -1=\delta (Y_k) (\geq \gamma(Y_k )). 
  \]
This proves that $(W_k, Y_k)$ satisfies the assumption (4) in \cref{cor: rel Diff Homeo ap}.
Applying \cref{cor: rel Diff Homeo ap}, we have that 
  \[
  \pi_n(\operatorname{Diff}  (W_k , \partial )) \to   \pi_n(\operatorname{Homeo}  (W_k , \partial ))
  \]
  is not an isomorphism for some $n \in \{0, \cdots 8\}$, and
  \[
  \pi_n(\operatorname{Diff}  (W_k )) \to   \pi_n(\operatorname{Homeo}  (W_k))
  \]
  is not an isomorphism for some $n \in \{0, \cdots 9\}$ by \cref{cor: Diff Homeo absolute ap}.
 \end{ex}

\section{Appendix}
 
In \cref{subsection: Proof of first main theo}, the proof of \cref{theo: main theo}, we use the following version of the equivariant Thom isomorphism several times. We give equivariant Thom isomorphism theorem with local coefficients. Although we use only equivariant cohomologies in coefficients $\F=\Z/2$,
Baraglia~\cite{Ba19} made use of equivariant cohomologies in local coefficient and used the Thom isomorphism of the form \cref{Thom}.

 Let $G$ be a compact Lie group. 
 Let $B$ be a paracompact Hausdorff space and $\pi_W: W\to B$ a $G$-vector bundle over $B$. Here we regard $B$ as a $G$-space with the trivial action.
Take $\rho$ be a $A$-valued local system on $B$ for a fixed Abelian group $A$. 

We define the local coefficient equivariant cohomology by 
\[
H^*_{G}(B; \rho ) := H^*( B\times BG; \pr^* \rho ), 
    \]
    where $\pr : B\times BG \to B$ is the projection. 
    
We first consider the vector bundle 
 \begin{align}\label{oo}
 p : W_{hG}:= EG \times_{G} W \to (EG \times B)/G := B_{hG}. 
\end{align}
 Note that 
$ ( EG \times D(W), EG \times S(W)) $ is $G$-homeomorphic to the pair $( D(EG \times W) , S(EG \times S(W)) )$. 
 This proves
 \[
 {H}^* ( EG \times_G D(W), EG \times_G S(W) ) \cong {H}^* (D(EG \times_G W) , S(EG \times_G S(W)) ). 
 \]
 Then, for any local system $\rho$ on $B$, we define the coefficient equivariant cohomology for the Thom space 
 \[
 {H}^{*}_{G} (D(W), S(W); \pi_W^*\rho  )
 \]
 \[
 :=  {H}^* (D(EG \times_G W) , S(EG \times_G S(W)) ; \pi_W^*\rho  ). 
 \]
 \begin{lem}We have the following isomorphisms. 
 \label{Thom} 
 \begin{itemize}
   \item[(i)] 
   The multiplication of an element 
\[
\tau_G(W) \in \tilde{H}^{\rank W}_{G} (\Th(W); \F )
\]
 gives an isomorphism 
  \[
  H^*_{G}(B; \F ) \to \widetilde{H}^{*+\rank W}_{G} (\operatorname{Th}(W); \F ). 
  \] 
 \item[(ii)] 
 Suppose $G$ is connected. 
The multiplication of an element 
\[
\tau_G(W) \in H^{\rank W} _{G} (D(W), S(W); \pi_W^*w_1 (W) )
\]
 gives an isomorphism 
  \[
  H^*_{G}(B; \rho  ) \to {H}^{*+\rank W}_{G} (D(W), S(W); \pi_W^*\rho \otimes  \pi_W^*w_1 (W ) ),
  \] 
  where $w_1(W)$ is the orientation local system of $W$. 
  \end{itemize}
 \end{lem}
 We give a sketch of proof of \cref{Thom}. 
 \begin{proof}For the $\F$-coefficient, the usual Thom isomorphism theorem implies that there exists an element $\tau_G(W) \in {H}^* (D(EG \times_G W) , S(EG \times_G S(W)) ; \F) )$ such that 
 \[
  \cup \tau_G(W): H^*( B\times BG; \F ) \to  {H}^{*+\rank W} (D(EG \times_G W) , S(EG \times_G S(W)) ; \F) 
 \]
 is an isomorphism. This proves (i).

 For the second statement, we use the local coefficient version of the Thom isomorphism theorem. 
An important point is that the orientation local system of \eqref{oo} is the same as the $\pr^* w_1(W)$.
 Then we have an element 
 \[
 \tau_G(W) \in {H}^* (D(EG \times_G W) , S(EG \times_G S(W)) ;  \pi_W^* w_1(W) )
 \]
  such that 
 \[
  \cup \tau_G(W): H^*( B\times BG; \pr^* \rho )
  \]
  \[
   \to  {H}^{*+\rank W}  (D(EG \times_G W) , S(EG \times_G S(W)) ; \pi_W^*\rho \otimes  \pi_W^* w_1(W) )
  \]
gives an isomorphism.
 \end{proof}

\bibliographystyle{plain}
\bibliography{tex}

\end{document}